\documentclass[a4paper]{amsart}
\usepackage{amssymb}
\usepackage{epsfig}
\usepackage{verbatim}
\usepackage{graphicx}

\let\phi\varphi

\def\Z{{\mathbb Z}}
\def\Q{{\mathbb Q}}
\def\R{{\mathbb R}}
\def\C{{\mathbb C}}
\def\E{{\mathbb E}}
\def\H{{\mathbb H}}

\def\F{{F}} 
\def\G{{G}} 
\def\comm{\mbox{\rm Comm}}
\def\vol{\mbox{\rm Vol}}
\def\SL{\mbox{SL}}
\def\minusmat{\left(\begin{array}{rr}-1&0\\0&-1\end{array}\right)}
\def\pmmat{\left(\begin{array}{rr}1&0\\0&-1\end{array}\right)}
\def\rflmat{\left(\begin{array}{rr}0&1\\1&0\end{array}\right)}
\def\trnmat{\left(\begin{array}{rr}0&1\\-1&0\end{array}\right)}
\def\n{{\mathbf n}}
\def\x{{\mathbf x}}
\def\P{{\mathcal P}}

\newtheorem{theorem}{Theorem}[section]
\newtheorem{lemma}[theorem]{Lemma}

\newtheorem{proposition}[theorem]{Proposition}

\newtheorem{remark}[theorem]{Remark}

\theoremstyle{definition}
\theoremstyle{remark}

\title{Commensurators of cusped hyperbolic manifolds}

\author{Oliver Goodman}

\author{Damian Heard}

\author{Craig Hodgson}

\address{Department of Mathematics and Statistics,
University of Melbourne, Parkville, Victoria 3010, Australia}
\email{oag@optusnet.com.au,~~damian.heard@gmail.com,~~cdh@ms.unimelb.edu.au}

\thanks{This work was partially supported by grants from the Australian Research Council}

\begin{document}

\begin{abstract}
This paper describes a general algorithm for finding the commensurator of a non-arithmetic hyperbolic manifold with cusps, and for deciding when two such manifolds are commensurable.  The method is based on some elementary observations
regarding horosphere packings and canonical cell decompositions.
For example, we use this to find the commensurators of all non-arithmetic hyperbolic once-punctured torus bundles over the circle. 

For hyperbolic 3-manifolds, the algorithm has been implemented using Goodman's computer program Snap.  We use this to determine the commensurability classes of all cusped hyperbolic $3$-manifolds triangulated using at most 7 ideal tetrahedra, and for the complements of hyperbolic knots and links with up to 12 crossings.
\end{abstract}

\maketitle

\section{Introduction}

Two manifolds or orbifolds $M$ and $M'$ are {\em commensurable} if they
admit a common finite sheeted covering. For hyperbolic
$n$-orbifolds, we can suppose that $M= \H^n/\Gamma$ and
$M'= \H^n/\Gamma'$, with $\Gamma$ and $\Gamma'$ discrete subgroups of
$\mbox{\rm Isom}(\H^n)$.
In this paper, we assume that $M$ and $M'$ are of finite volume
and of dimension at least $3$. Then, by Mostow-Prasad Rigidity,
commensurability means that we can conjugate
$\Gamma$ by an isometry $g$ such that $g\Gamma g^{-1}$ and $\Gamma'$
intersect in a subgroup of finite index in both groups.

Given that the classification of finite
volume hyperbolic manifolds up to homeomorphism appears to be
hard, it seems sensible to attempt to subdivide the problem and start
with a classification up to commensurability. Looked at in this way,
we see a remarkable dichotomy between the {\em arithmetic} and
{\em non-arithmetic} cases. (See \cite{MacReid} for the definition of arithmetic
hyperbolic manifolds.)

Define the {\em commensurator} of $\Gamma$ to be
the group
$$\mbox{\rm Comm}(\Gamma) = \{g\in \mbox{\rm Isom}(\H^n) \mid
[\Gamma:\Gamma\cap g\Gamma g^{-1}]<\infty\}.$$ 
Then $\Gamma$ and $\Gamma'$ are commensurable if and only if
$\comm(\Gamma)$ and $\comm(\Gamma')$ are conjugate.
Geometrically, an element of the normalizer of $\Gamma$ in ${\rm Isom}(\H^n)$
represents a {\em symmetry} (i.e. isometry) of $M=\H^n/\Gamma$. Similarly,
an element of the commensurator represents an isometry
between finite sheeted covers of $M$; this gives a 
{\em hidden symmetry} of $M$ if it is not the lift of an isometry of $M$ (see \cite{NR}).

It follows from deep work of Margulis~\cite{Marg} (see also \cite{Zi}), that in
dimension $\ge 3$,  the commensurator
$\comm(\Gamma)$ is discrete if and only if
$\Gamma$ is not arithmetic. This means that the commensurability class
of a non-arithmetic, cofinite volume, discrete group $\Gamma$ is
particularly simple, consisting only of conjugates of the finite
index subgroups of $\comm(\Gamma)$. In terms of orbifolds, it means
that $M$ and $M'$ are commensurable if and only if they cover a common
quotient orbifold.

On the other hand, commensurability classes of
arithmetic groups are ``big'': we may well have commensurable $\Gamma$
and $\Gamma'$ such that the group generated by $g\Gamma g^{-1}$ 
and $\Gamma'$ is not discrete for any $g$.

A hyperbolic $n$-orbifold is {\em cusped} if it is non-compact of finite volume. 
This paper describes a practical algorithm for determining when two
cusped hyperbolic $n$-manifolds cover a common quotient, and for
finding a smallest quotient. For non-arithmetic finite volume {\em cusped}
hyperbolic $n$-manifolds of dimension $n \ge 3$, this solves the commensurability problem. 

Section 2 begins with some elementary observations about horoball packings 
and canonical cell decompositions of a cusped hyperbolic manifold. This leads to a characterization of the commensurator of a non-arithmetic cusped hyperbolic manifold $M$ as the maximal symmetry group of the tilings of $\H^n$ obtained by lifting canonical cell decompositions of $M$. In Section 3, we use this to determine the commensurators of non-arithmetic  hyperbolic once-punctured torus bundles over the circle. 

 Section 4 gives an algorithm for finding the isometry group of a tiling of $\H^n$ arising from a cell decomposition of a hyperbolic manifold, and Sections 5 and 6 describe methods for finding all possible canonical cell decompositions for a cusped hyperbolic manifold. Section 7 contains some observations on commensurability of
  cusps in hyperbolic 3-manifolds which can simplify the search for all canonical cell decompositions. 

In 3-dimensions, 
each orientable hyperbolic orbifold has the form $M=\H^3/\Gamma$, where $\Gamma$ is a discrete subgroup of $\mbox{PSL}(2,\C) = \mbox{\rm Isom}^+(\H^3)$.
The {\em invariant trace field} $k(\Gamma)\subset \C$ is the field generated
by the traces of the elements of $\Gamma^{(2)} = \{\gamma^2\mid
\gamma\in\Gamma\}$ lifted to
$\mbox{SL}(2,\C)$. This is a number field if $M$ has finite volume (see \cite{NR}, \cite{R}, \cite{Ma}). The {\em invariant quaternion
algebra} is the $k(\Gamma)$ subalgebra of $M_2(\C)$ generated by
$\Gamma^{(2)}$. These are useful and
computable commensurability invariants (see \cite{CGHN}, \cite{MacReid}).

For the {\em arithmetic} subgroups of $\mbox{\rm Isom}(\H^3)$, 
the invariant quaternion algebra is a complete commensurability invariant.
In fact for {\em cusped} arithmetic hyperbolic 3-orbifolds, 
the invariant trace field is an imaginary quadratic field and the quaternion
algebra is just the algebra of all $2 \times 2$ matrices
with entries in the invariant trace field (see \cite[Theorem 3.3.8]{MacReid}); so the invariant trace field
is a complete commensurability invariant.
However most cusped hyperbolic $3$-manifolds are non-arithmetic (cf. \cite{Bor}) 
so other methods are needed to determine commensurability.

Damian Heard and Oliver Goodman have implemented  
the algorithms described in this paper for non-arithmetic hyperbolic 3-manifolds;
these are incorporated in the computer program Snap \cite{Snap}. 
Using this we have determined the commensurability classes for all manifolds
occurring in the Callahan-Hildebrand-Weeks census (\cite{HiW}, \cite{CHW}) of 
cusped hyperbolic manifolds with up to $7$ tetrahedra, and for complements of hyperbolic knots and links
up to 12 crossings, supplied by Morwen Thistlethwaite (see \cite{knotscape}).
These results are discussed in Section 8, while Section 9 outlines the Dowker-Thistlethwaite notation used to describe links.

This work has uncovered interesting new examples of commensurable knot and link complements (see Examples \ref{5_links} and \ref{comm_knots}), and a new example of a knot with shape field properly contained in the invariant trace field (see Example \ref{shape_not_trace}). The results have also been used by Button \cite{Bu} to study 
fibred and virtually fibred cusped hyperbolic 3-manifolds.

For $1$-cusped manifolds we note that ``cusp density'' (see
Section~\ref{horopack}) is a very good
invariant. We have found only a few examples of incommensurable
$1$-cusped manifolds which are not distinguished by cusp density (see
Example~\ref{hball_diff}). 

There is also a ``dumb'' algorithm, based on volume bounds for
hyperbolic orbifolds, which works for any (possibly closed) 
non-arithmetic hyperbolic 3-orbifold, but
appears to be quite impractical. If $M$ and $M'$ cover $Q$ with
$\vol(Q) > C$ then the degrees $d,d'$ of the coverings are
bounded by $D = \lfloor\vol(M)/C\rfloor$ and $D' =
\lfloor\vol(M')/C\rfloor$, respectively. Then if $M$ and
$M'$ are commensurable, they admit a common covering $N$ of degree at
most $D'$ over $M$ and at most $D$ over $M'$. The best current
estimate for $C$ for orientable non-arithmetic 3-orbifolds is $0.041\ldots$
from recent work of Marshall-Martin \cite{MM}. Since $\vol(M)\approx 2$
is typical, we would have to find all
coverings of $M'$ of degree $d'\leq 50$. This means finding all
conjugacy classes of transitive representations of $\pi_1(M')$ into
$S_{50}$, a group with around $10^{64}$ elements!

Acknowledgements: We thank Ian Agol for pointing out a simplification 
to our method of determining the commensurability of Euclidean tori, 
Gaven Martin for information on current volume bounds,
and Walter Neumann for several interesting discussions on this work. 
We also thank Alan Reid, Genevieve Walsh, 
and the referee for their helpful comments on the paper.

\section{The Commensurability Criterion}\label{horopack}

We use the following terminology throughout this paper.
A set of disjoint horoballs in $\H^n$ is called a {\em horoball packing}, and 
a  {\em cusp neighbourhood} in a hyperbolic $n$-orbifold is one which lifts to such a horoball packing.

\begin{lemma} \label{hdisc}
  The symmetry group of a horoball packing in $\H^n$ is discrete
  whenever the totally geodesic subspace spanned by their ideal points
  has dimension at least $n-1$. 
\end{lemma}

\begin{proof}
  Let $\{g_i\}$ be a sequence of symmetries of the packing converging
  to the identity. Choose horoballs $B_1,\ldots,B_n$ whose ideal
  points span a totally geodesic subspace $H$ of dimension $n-1$.  For
  $i$ sufficiently large, we can assume that $g_i(B_k) = B_k$ for
  $k=1,\ldots,n$. But this implies that these $g_i$ fix $H$
  pointwise. Since the only such isometries are the identity, and
  reflection in $H$, the sequence must be eventually constant. 
\end{proof}

\begin{lemma} \label{ispan}
  Let $M=\H^n/\Gamma$ be a finite volume cusped hyperbolic orbifold.
  The set of parabolic fixed points of $\Gamma$ spans $\H^n$.
\end{lemma}

\begin{proof}
  The set of parabolic fixed points is dense in the limit set of
  $\Gamma$ which equals the whole of the sphere at infinity.
\end{proof}

\begin{lemma} 
  Let $M$, $M'$ be finite volume cusped hyperbolic orbifolds. Then $M$
  and $M'$ cover a common orbifold $Q$ if and only if they admit
  choices of cusp neighbourhoods lifting to isometric horoball packings.
\end{lemma}

\begin{proof}
If $M$ and $M'$ cover $Q$, choose cusp neighbourhoods in $Q$ and lift to
$M$ and $M'$. These all lift to the same horoball packing in $\H^n$,
namely the horoball packing determined by our choice of cusp
neighbourhoods in $Q$. Conversely, both $M$ and $M'$ cover the quotient of
$\H^n$ by the group of symmetries of their common horoball packing
which, by Lemmas~\ref{hdisc} and \ref{ispan}, is discrete. 
\end{proof}

We can define the cusp density of a {\em $1$-cusped} 
hyperbolic orbifold $M$ as follows. Since $M$ has only one cusp it has
a unique maximal (embedded) cusp neighbourhood $U$. The {\em cusp density} of
$M$ is $\vol(U)/\vol(M)$. 
Since the cusp density of any orbifold
covered by $M$ is the same, it is a commensurability invariant of
orbifolds with discrete commensurator. 

Choosing a full set of disjoint cusp neighbourhoods in a non-compact finite volume
hyperbolic $n$-manifold $M$ determines a ``Ford spine.'' This is the
cell complex given by the set of points in $M$ equidistant from the
cusp neighbourhoods in two or more directions.  Cells of dimension
$n-k$ contain points equidistant from the cusp neighbourhoods in $k+1$
independent directions ($k=1, \ldots, n$). This spine can also be seen
intuitively as the ``bumping locus'' of the cusp neighbourhoods: blow
up the cusp neighbourhoods until they press against each other and
flatten.

Dual to the Ford spine is a decomposition of $M$ into ideal polytopes, 
generically simplices. The ideal cell dual to a given $0$-cell of the
Ford spine lifts to the convex hull in $\H^n$ of the set of ideal points
determined by the equidistant directions. We call the cell
decompositions that arise in this way {\em canonical.}\footnote{The
  term is not really ideal since, for manifolds with multiple cusps,
  there are generally multiple canonical cell decompositions depending
  on the choice of cusp neighbourhoods.}  For a 1-cusped manifold the
canonical cell decomposition is unique.  It is shown in \cite{Ak} that
a finite volume hyperbolic manifold with multiple cusps admits
finitely many canonical cell decompositions.

\begin{theorem} \label{comm_criterion}
Cusped hyperbolic $n$-manifolds $M$ and $M'$ cover a
common orbifold if and only if they admit canonical ideal cell
decompositions lifting to isometric tilings of $\H^n$.
\end{theorem}

\begin{proof}
If $M$ and $M'$ cover $Q$, choose cusp neighbourhoods  
in $Q$ and lift them to
$M$, $M'$ and $\H^n$. Constructing the Ford spine and cell
decomposition in $\H^n$ clearly yields the lifts of those entities
from both $M$ and $M'$ corresponding to our choice of
cusp neighbourhoods.

Conversely, observe that the symmetry group of the common tiling
gives an orbifold which is a quotient of both manifolds.
\end{proof}

\begin{remark} We can omit the word `canonical' in the above
theorem. The proof is unchanged.
\end{remark}

The previous theorem gives the following characterization of the commensurator.

\begin{theorem} \label{commens=maxsymm}
  Let $M = \H^n/\Gamma$ be a finite volume cusped hyperbolic
  $n$-manifold with discrete commensurator. Then $\comm(\Gamma)$ is
  the maximal symmetry group of the tilings of $\H^n$ obtained by lifting 
  canonical cell decompositions of $M$; it contains all such symmetry groups.
\end{theorem}

For manifolds with discrete commensurator, we can now define a {\em
truly canonical} ideal cell decomposition as follows. Find
$\comm(\Gamma)$ as in the above theorem. Choose 
{\em equal volume} cusp neighbourhoods
in $\H^n/\comm(\Gamma)$. Lift them to $M$ and take the
resulting canonical cell decomposition of $M$. Two such manifolds are
commensurable if and only if their truly canonical cell decompositions
give isometric tilings of $\H^n$. 

Choosing maximal cusp neighbourhoods in $\H^n/\comm(\Gamma)$ also  gives a
canonical version of {\em cusp density} for multi-cusped manifolds.

Theorem  \ref{commens=maxsymm} is the basis for the algorithms described in this paper. Canonical cell decompositions can be computed by the algorithms of 
Weeks described in \cite{We1} and implemented in SnapPea \cite{We}. 
  In Section 4 below we give an algorithm for finding the isometry groups of the corresponding tilings of $\H^n$. Combining this with Theorem 
  \ref{commens=maxsymm} gives an algorithm for finding commensurators of 
  {\em 1-cusped} non-arithmetic hyperbolic $n$-manifolds.
In Sections 5 and 6 we extend this algorithm to {\em multi-cusped} manifolds, 
by describing methods for finding {\em all} canonical cell decompositions. 

For hyperbolic 3-manifolds, these algorithms have been implemented by Heard and Goodman.  (These are incorporated in ``{\tt find commensurator}'' and related commands in the program Snap \cite{Snap}).
We conclude this section with some examples discovered during this work.

\subsection{Example: a 5-link chain and friends}\label{5_links}
The following five links have commensurable complements, as shown using our computer program. 
In the first three cases at least it is possible to `see' this
commensurability. 
$$
\includegraphics[height=2in]{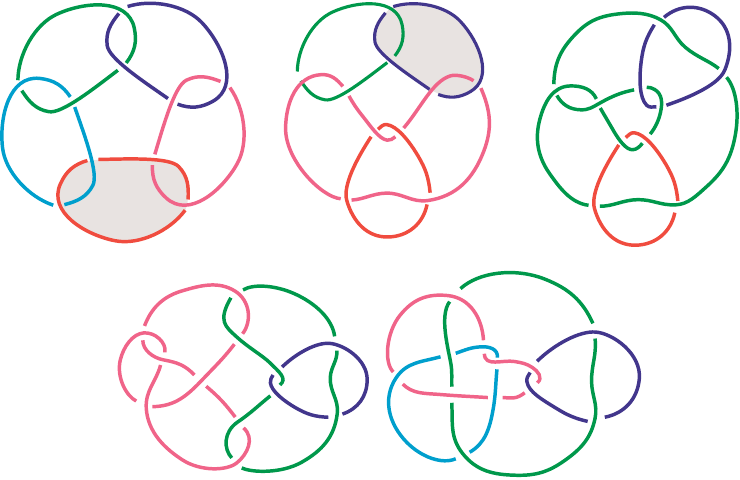}
$$
The first of these is the $5$-link chain $C_5$. Thurston 
\cite[Chapter 6]{Th} explains how to obtain a fundamental region for
the hyperbolic $k$-link chain complements: we span each link of the
chain by a disk in the obvious manner. The complement of the union of these
five disks is then
a solid torus. Once the link is deleted, 
the disks and their arcs of intersection divide the boundary of the solid 
torus into ideal squares $A,B,C,D,E$ as shown, with cusps labelled $a,b,c,d,e$. 

$$
\includegraphics[height=1.5in]{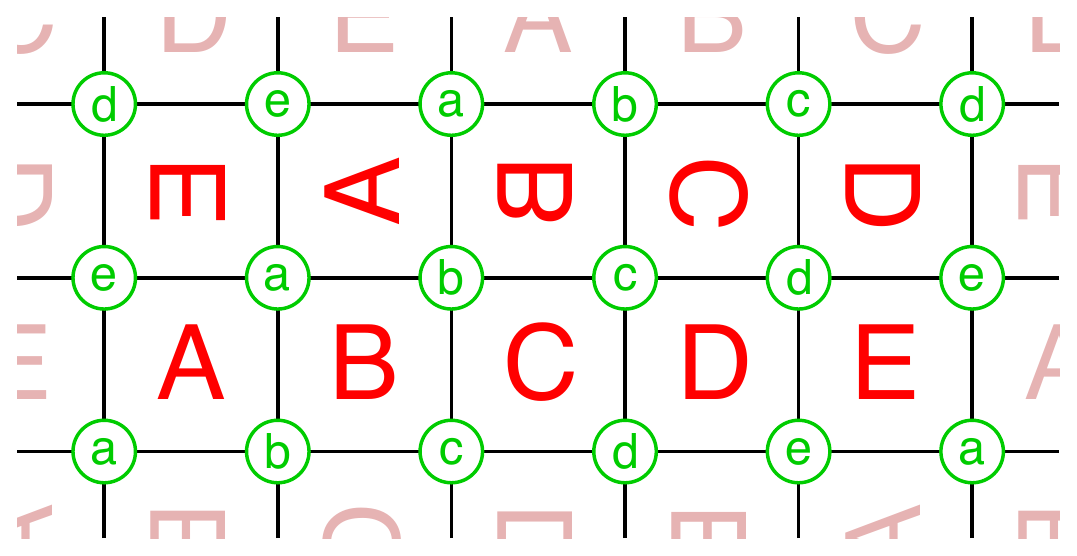}
$$

A hyperbolic structure is given by taking two regular pentagonal drums
with ideal vertices and adjusting their heights to obtain (ideal) square
faces. Glue two drums together as shown, identify the top with
the bottom via a $\frac{4\pi}{5}$ rotation, and glue faces as
indicated. Edges are then identified in $4$'s, two horizontal with two
vertical. It is easy to check that the sum of dihedral angles 
around each edge is $2\pi$ so this gives a hyperbolic structure, since the angle sum is $\pi$ at each ideal vertex of a drum.
$$
\includegraphics[height=2in]{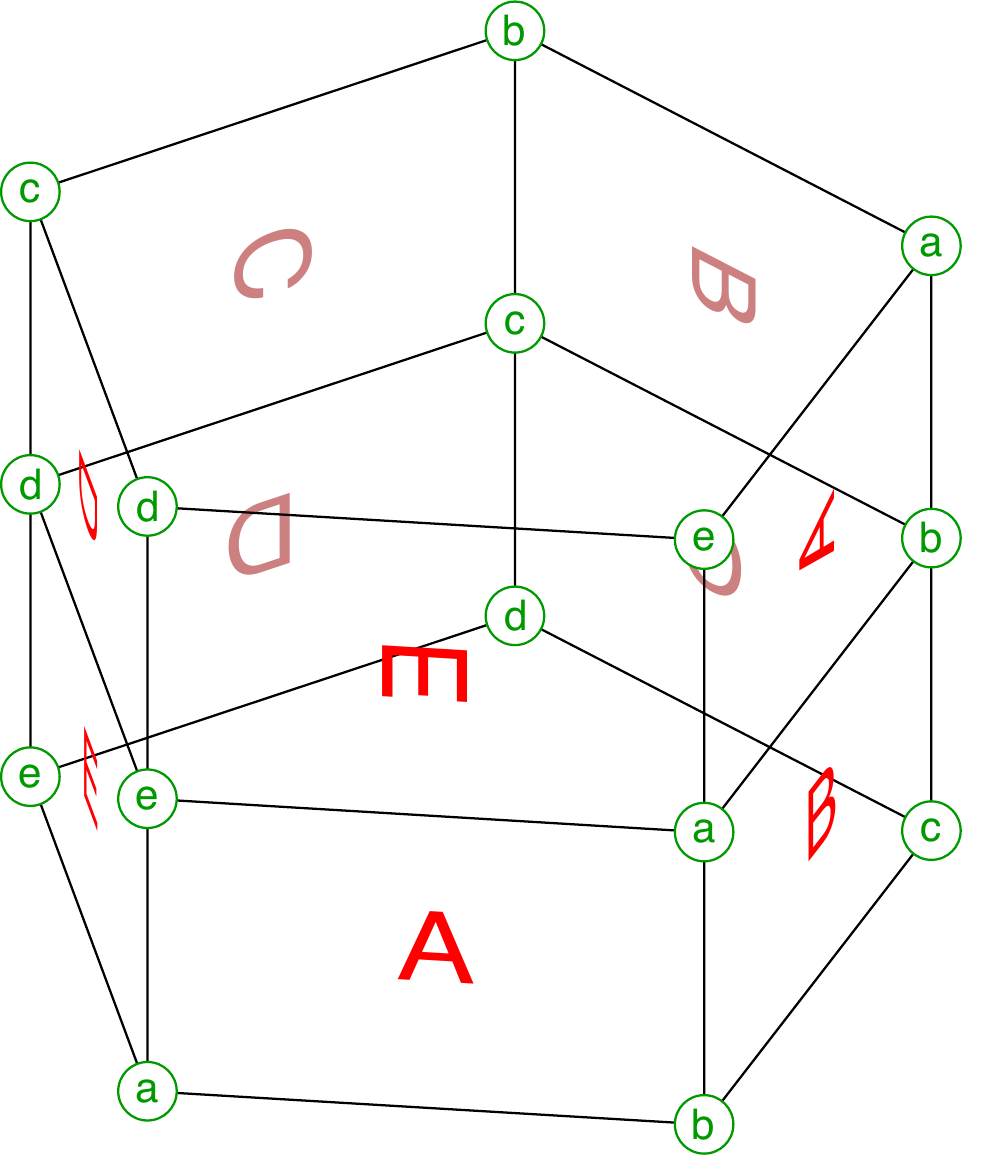}
$$
It is clear from the symmetry of the picture that the drums are cells
in the canonical cell decomposition obtained by choosing equal area
cusp cross sections. We remark that Neumann-Reid show that this link complement
is non-arithmetic in \cite[Section 5]{NR}.

Now change the $5$-link chain by cutting along the shown disk,
applying a half twist, and re-gluing to obtain our second link. 
In the complement, this surgery introduces a half turn into the gluing
between the $A$-faces. Edges are still identified in $4$'s, 
two horizontal with
two vertical, but there are now only four cusps. 
$$
\epsfig{file=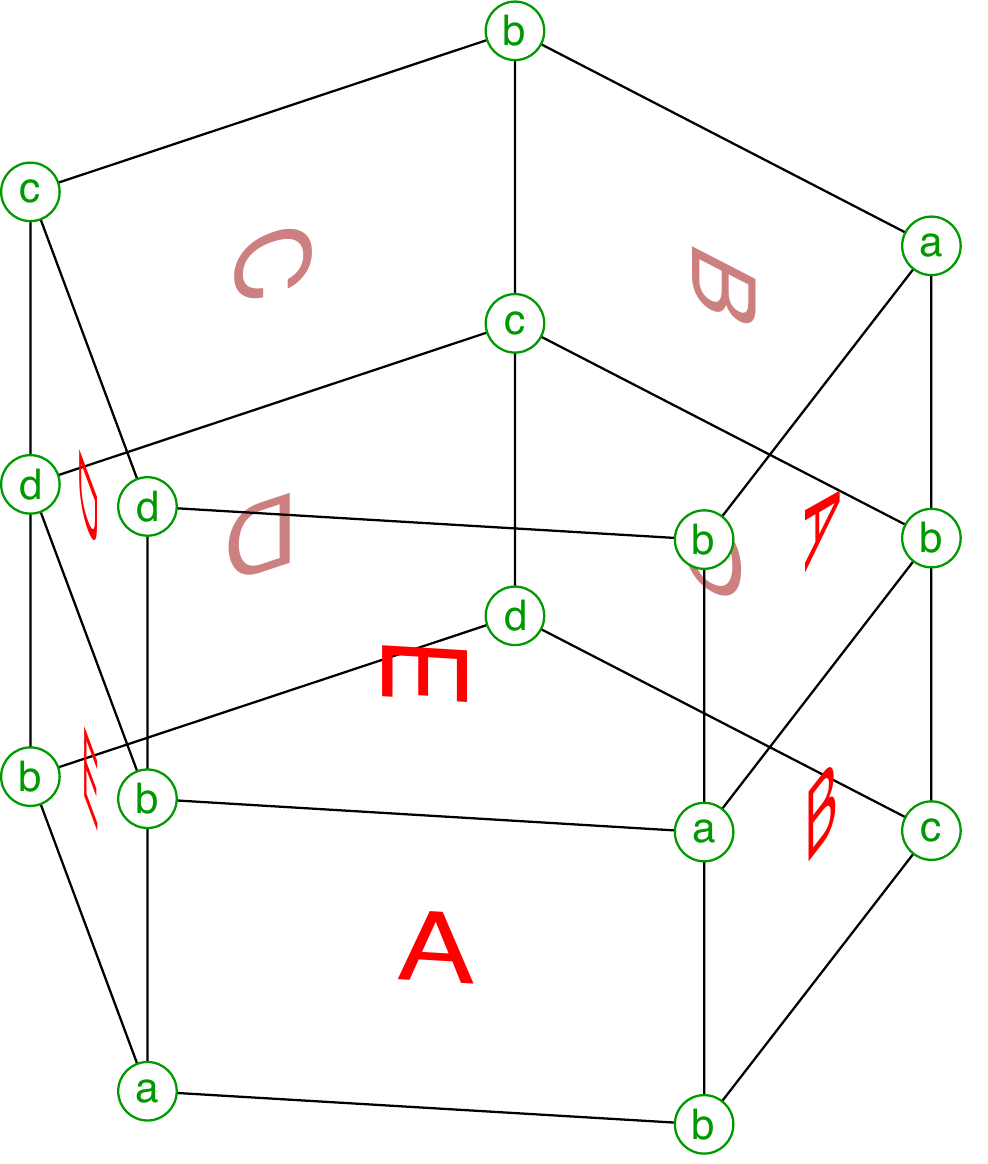, height=2in}
$$
If we repeat the process on the second disk shown we obtain the third
link. Again this corresponds to changing the gluing pattern
on our two drums. 

Since these link complements are non-arithmetic, the tiling of $\H^3$
by pentagonal drums covers some canonical cell decomposition of each
one.  Since their volumes are the same, each one decomposes into two
pentagonal drums.  It should therefore be possible, in each case, to
find $5$ ideal squares meeting at order $4$ edges, cutting the
complement into one or two solid tori. We leave this as a challenge
for the reader. 

\subsection{Example: commensurable knot complements}\label{comm_knots}
Commensurable knot complements seem to be rather rare.  Previously
known examples include the Rubinstein-Aitchison dodecahedral knots
\cite{AR} and examples due to a construction of Gonz\'{a}les-Acu\~{n}a
and Whitten~\cite{GW} giving knot complements covering other knot
complements. For example, the $-2,3,7$ pretzel knot has $18/1$ and
$19/1$ surgeries giving lens spaces. Taking the universal covers of
these lens spaces gives new hyperbolic knots in $S^3$ whose complements are $18$-
and $19$-fold cylic covers of the $(-2,3,7)$--pretzel complement.

Our program finds a pair of knots ``9n6'' and ``12n642'', having 9 and 12 crossings
respectively, whose complements are commensurable with volumes in the
ratio $3:4$. 
Walter Neumann has pointed out that these knots belong to
a very pretty family of knots: take a band of $k$ repeats of a trefoil
with the ends given $m$ half twists before putting them
together. E.g. $(k,m) = (3,2)$: 
\[\epsfig{file=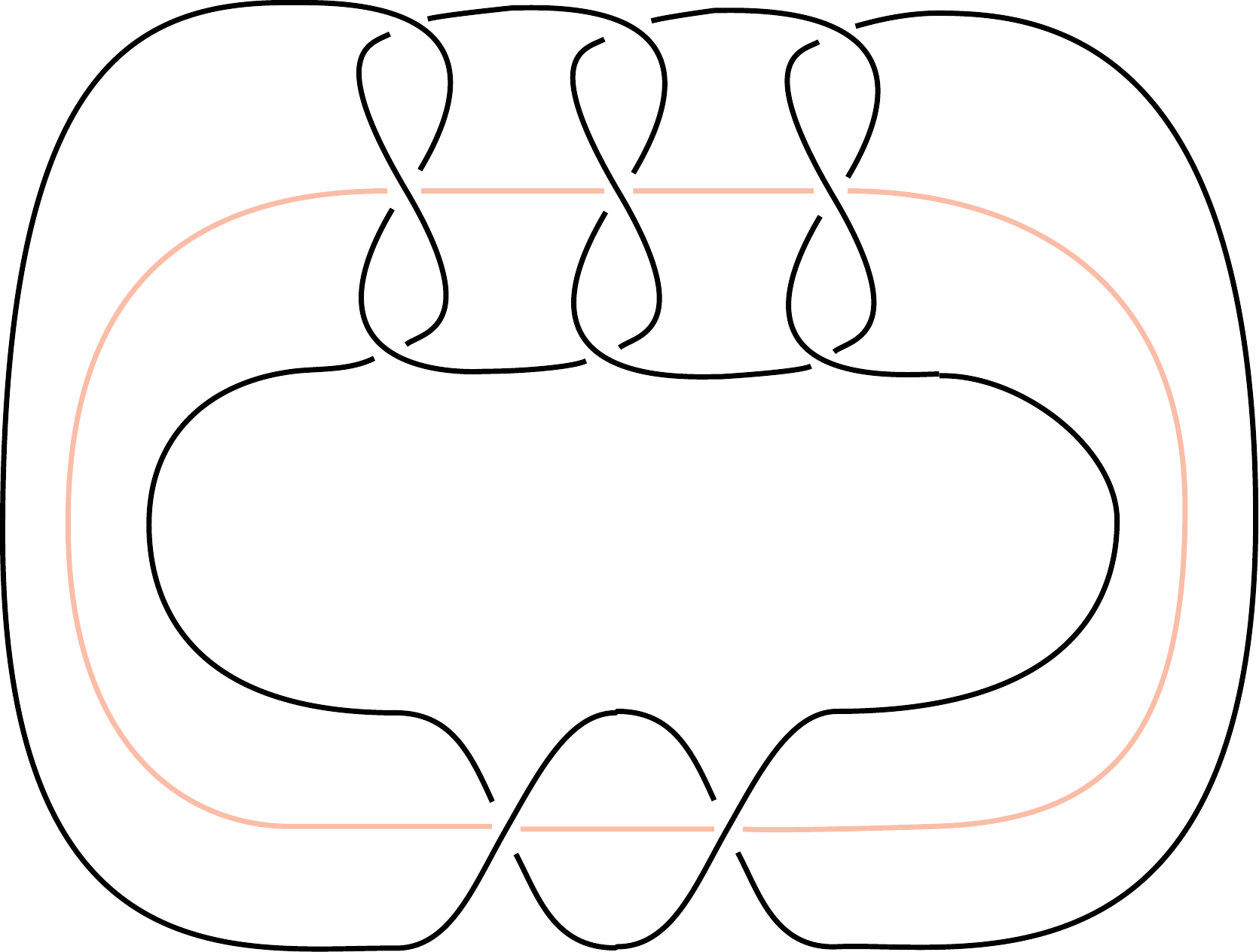, height=4cm}\] 
The half-twists are put in so as to undo some of the crossings
of the trefoils (allowing the above projection to be rearranged
so as to have 9 crossings). The pair of knots found by our program
correspond to $(k,m) = (3,2)$ and $(4,1)$. 

To see that these knots have commensurable complements we find a
common quotient orbifold. In each case this is the quotient of the
knot complement by its symmetry group; 
these are dihedral groups of order 12 and order 16 respectively.

The picture of $9n6$ above shows an
obvious axis of $2$-fold symmetry; below left is the quotient, which is the
complement of a knot in the orbifold $S^3$ with singular set an unknot labelled $2$.
By pulling the knot
straight, we see that this is an orbifold whose underlying space is a
solid torus with knotted singular locus.
\[\mbox{
  \begin{picture}(0,0)(0,0)
    \put(123,65){$\scriptstyle 2$} 
  \end{picture}}
\epsfig{file=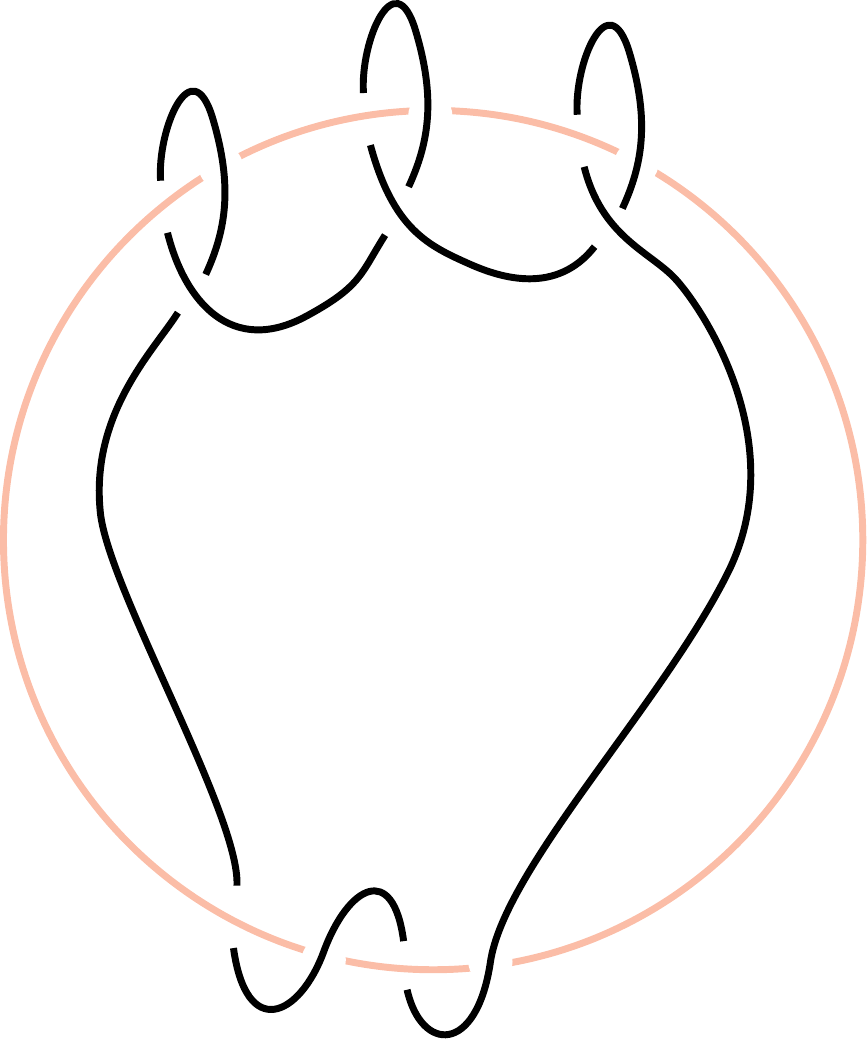, height=3.5cm}\hspace{2cm}
\mbox{
  \begin{picture}(0,0)(0,0)
    \put(138,65){$\scriptstyle 2$} 
  \end{picture}}
\epsfig{file=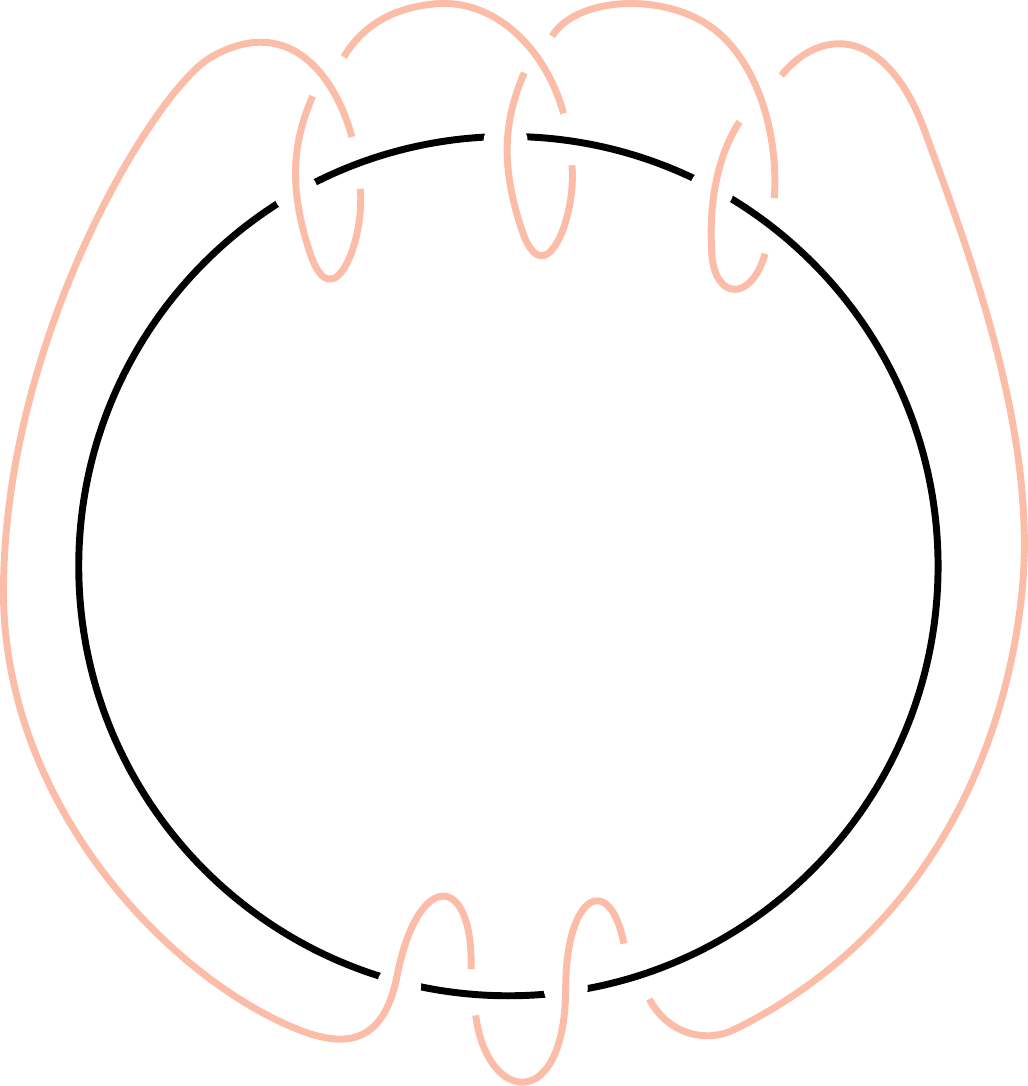,height=3.5cm}
\]

If we arrange the singular locus on a torus parallel to the boundary of
the solid torus we see 3 clasps, 3 strands in the (vertical) core
direction, and a strand with slope $2/1$. The view from inside the solid torus looking towards the boundary is shown below.
(For the knot $(4,1)$ we
would see 4 clasps, 4 strands in the core direction and a strand with
slope $1/1$.)
\[\epsfig{file=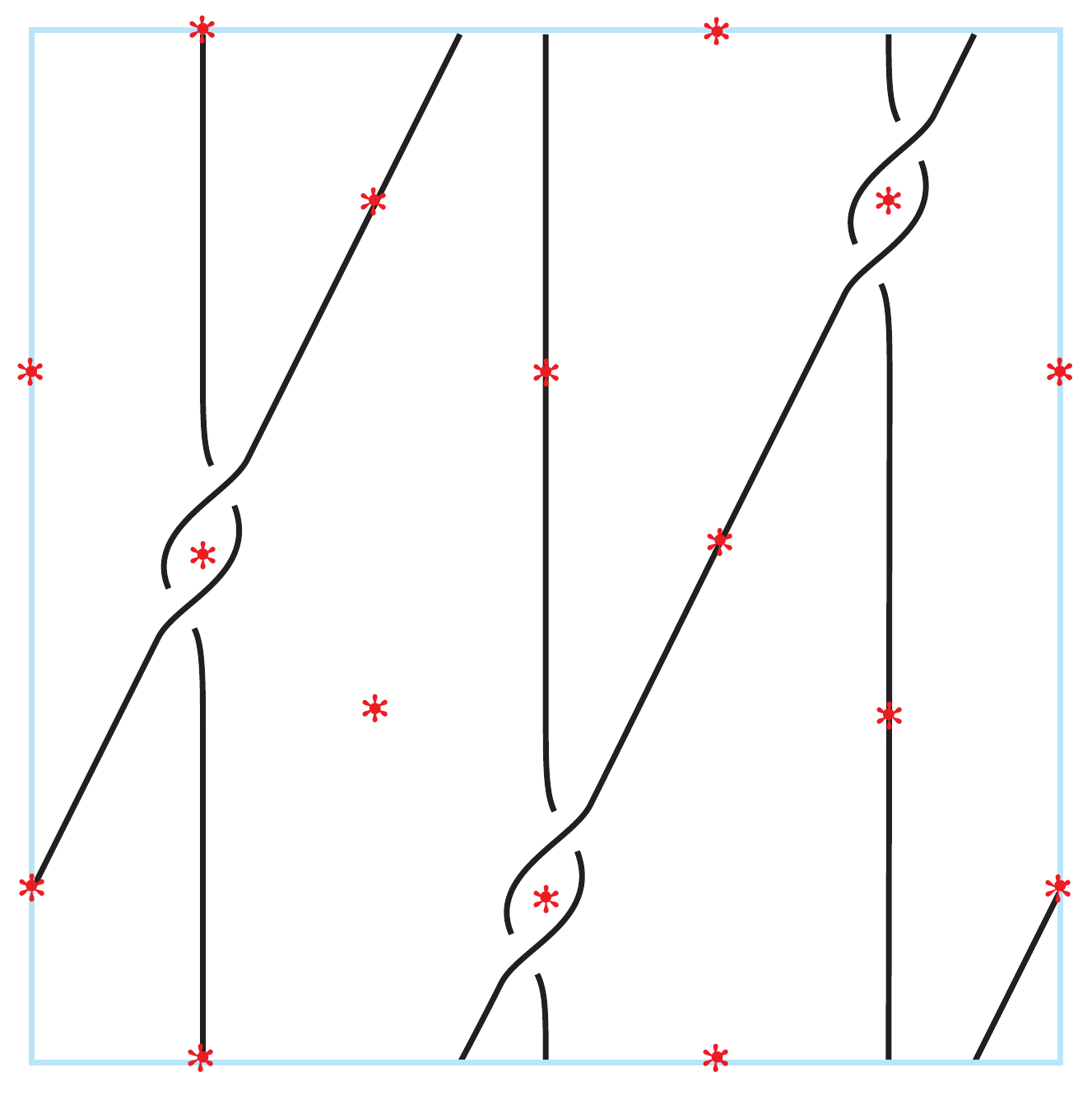, height=4cm}\]
The solid torus with its singular set has three 2-fold symmetries 
whose axes intersect the solid torus in 6 arcs,
each passing perpendicularly through the core like a skewer, with
symmetry group dihedral of order 6. The ends of the arcs are shown as stars
above. 

The quotient orbifold is obtained
by taking a slice of the solid torus between two axes and folding
closed the top and bottom disks like books. The result is a ball with
the axes giving two unknotted arcs of order 2 in the singular set,
running out to the boundary (which is now a $(2,2,2,2)$--pillowcase
orbifold). The original singular set gives an arc linking the other
two, so that the whole singular locus is an `H' graph labelled with
2's. 

The last three pictures show what happens to the singular locus in one
slice of the solid torus as we fold. We begin with the annulus in the bottom 1/6th
of the previous figure, redrawn after twisting the bottom. This bounds a solid
cylinder with the singular locus as shown in the middle figure. Folding along the top and bottom (and expanding the region slightly) gives the final result.
\[\raisebox{.5cm}{\epsfig{file=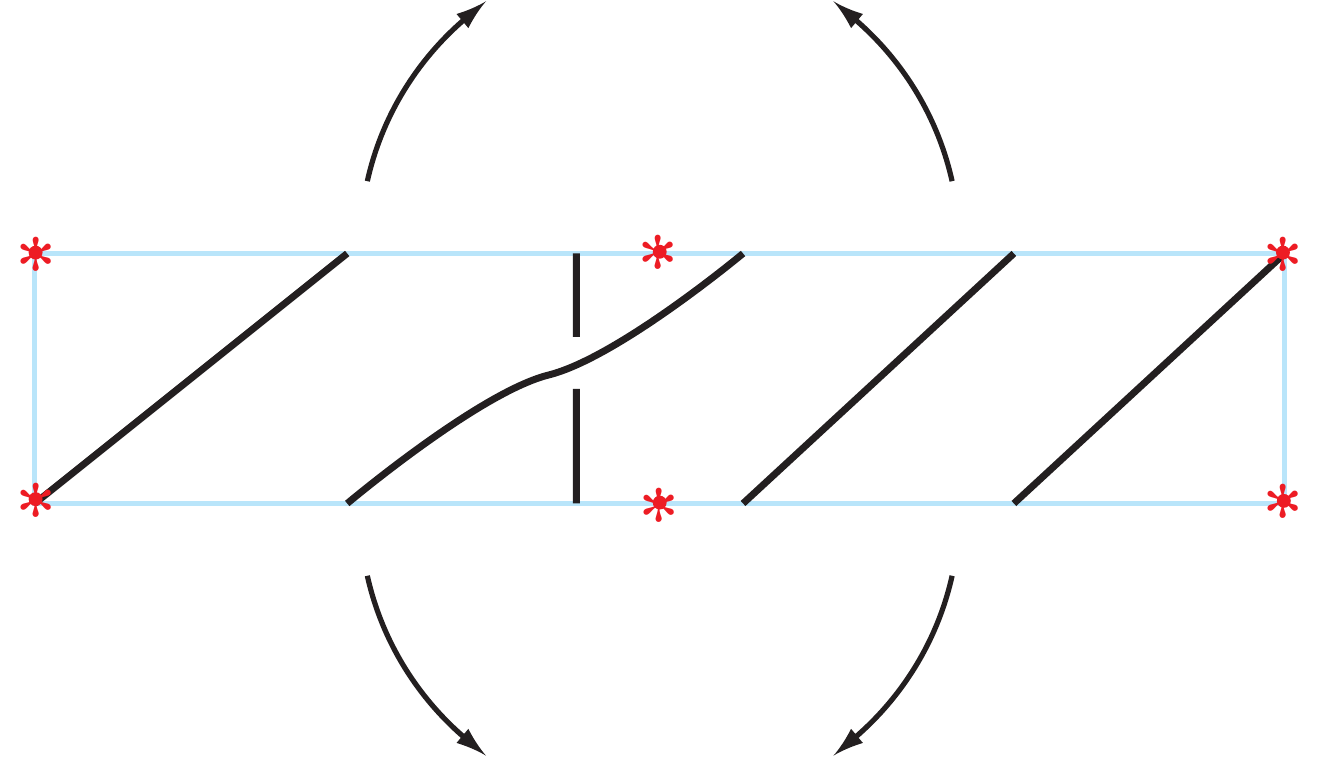, width=4.8cm}}\hspace{.45cm}
\raisebox{.25cm}{\epsfig{file=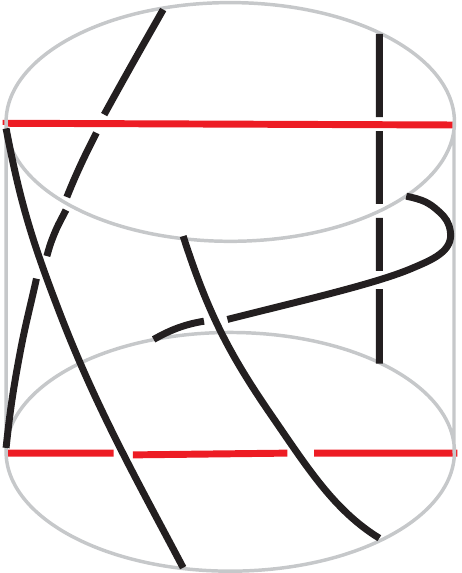, height=3.5cm}}\hspace{.5cm}
\epsfig{file=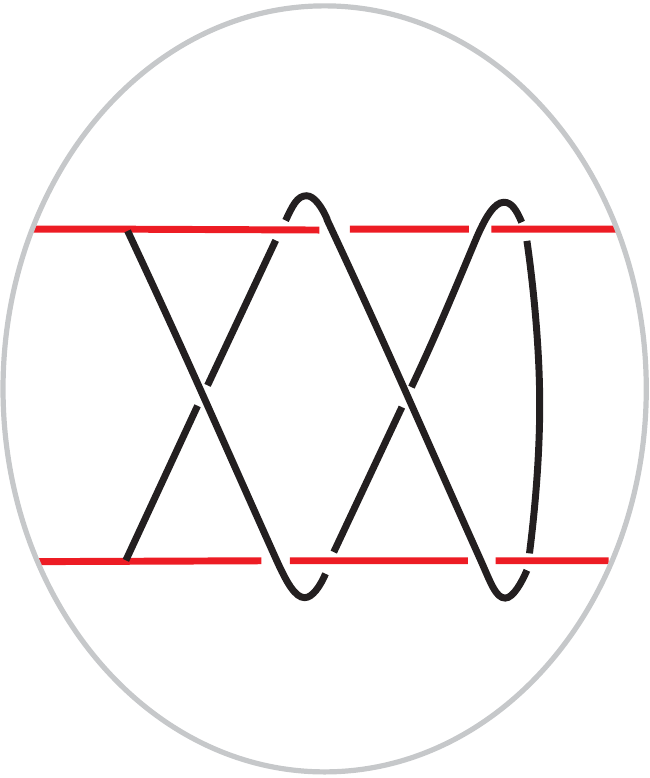,height=4cm}
\]
We leave it to the reader to draw similar pictures for the knot with
$(k,m) = (4,1)$ and verify that the result is indeed the same
orbifold. Alternatively, Orb \cite{Orb} or SnapPea \cite{We} can be used to verify that
the appropriate dihedral covers of the final orbifold give the complements of
the knots with $(k,m) = (3,2)$ and $(4,1)$.

We remark that Walter Neumann has found an infinite family of new examples of pairs of commensurable knot complements in the $3$-sphere; this example is the simplest case.

\subsection{Example: cusp horoball pictures} \label{hball_diff}
Figure~\ref{hball_different} shows the horoball packings of two
1-cusped census manifolds {\tt m137} and {\tt m138} as seen from the cusp. 
Using Snap  \cite{Snap}, we find that the commensurability
classes of these two equal-volume manifolds are indistinguishable by
cusp density or invariant trace field. Their maximal horoball packings
and canonical cell decompositions are however different.
(For example, the edges joining degree 4 vertices in the following cusp diagrams
are all parallel for m137, but not for m138.)

\begin{figure}[h] 
\epsfig{file=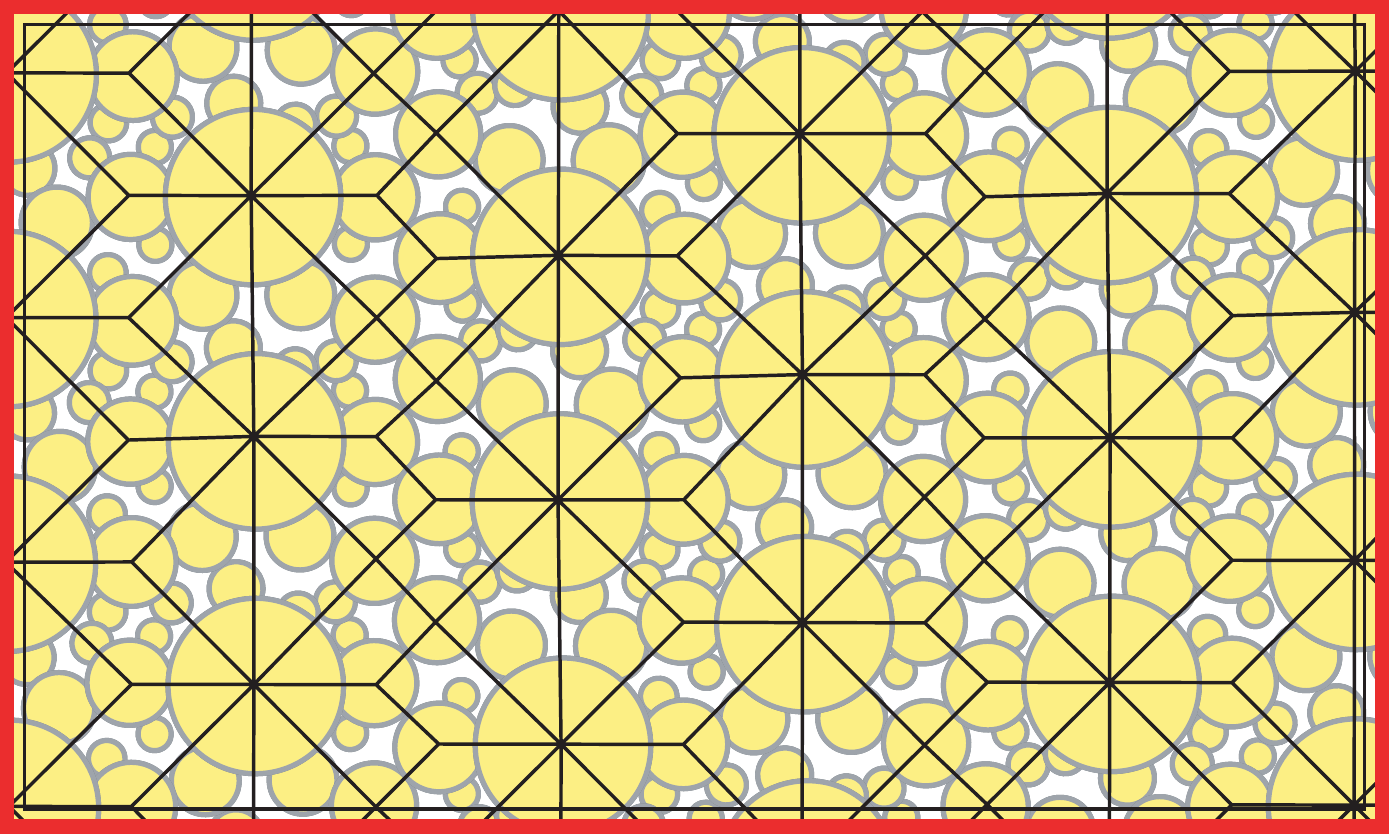, height=1.8in} 
\epsfig{file=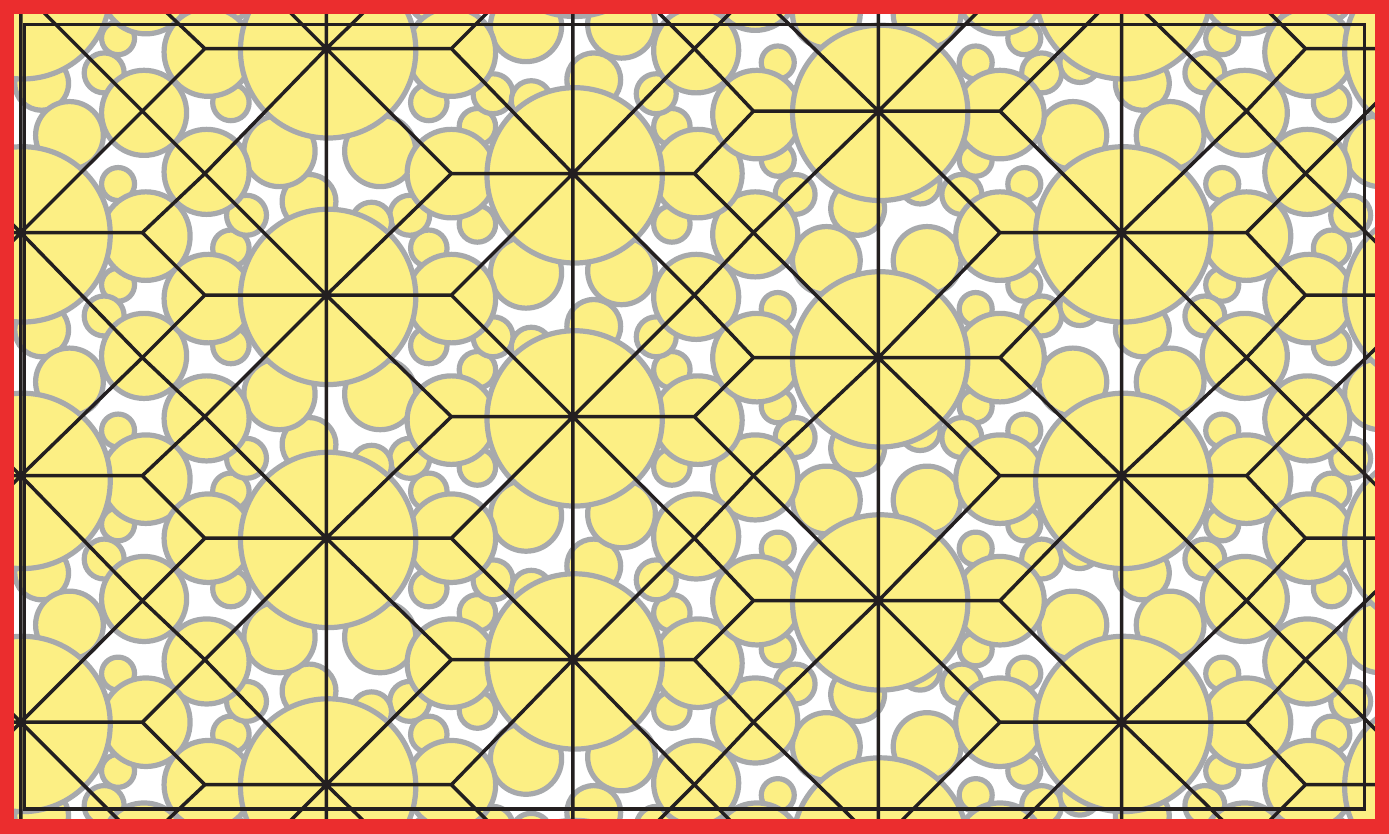, height=1.8in}
\caption{ 
The maximal horoball packing and canonical
cell decomposition as seen from the cusps of 
manifolds {\tt m137} and {\tt m138}.}
\label{hball_different}
\end{figure}

\section{Example: punctured torus bundles}
The bundles over $S^1$ with a once-punctured torus as fibre provide an interesting 
family of 1-cusped hyperbolic 3-manifolds with known
canonical triangulations. By analysing how symmetries of the lifted
triangulations of the universal cover appear 
when viewed from the cusp we obtain (in all the non-arithmetic cases)  
strong constraints on what symmetries may be possible. 
We then show that they all come from symmetries of the manifold.  
This leads to the following:

\begin{theorem} \label{nohs}
Let $M=\H^3/\G$ be an orientable non-arithmetic hyperbolic $3$-manifold 
which is a once-punctured torus bundle over $S^1$. Then $M$ has no ``hidden symmetries'',
i.e. the commensurator of $\G$ is the normalizer of $\G$ in $\mbox{\rm Isom}(\H^3)$.
\end{theorem}

Let $\F$ denote a once-punctured torus, and let $\phi:\F\rightarrow \F$ be
an orientation preserving homeomorphism. Let 
\[
M = M_\phi = \F \times_\phi S^1 = \frac{\F\times [0,1]}{(x,0)\sim (\phi(x),1)}
\]
be the mapping torus of $\phi$. 
Identifying $\F$ with $(\R^2 - \Z^2)/\Z^2$, we have that $\phi$ is
isotopic to an element of $\mbox{SL}(2,\Z)$; since $M$ depends only
on the isotopy class of $\phi$ we assume $\phi\in\mbox{SL}(2,\Z)$.
Then $M$ is hyperbolic whenever 
$\phi$ is hyperbolic, i.e. when $\phi$ has distinct real eigenvalues;
$M_\phi$ and $M_{\phi'}$ are homeomorphic if and only if $\phi$ and $\phi'$ are
conjugate. 

Define matrices \[
L = \left(\begin{array}{rr} 1 & 0\\ 1 & 1 \end{array}\right),\hspace{1em}
R = \left(\begin{array}{rr} 1 & 1\\ 0 & 1 \end{array}\right).
\]
For each word $w$ in the symbols $L,R$
define $\phi_w\in\SL(2,\Z)$ as the corresponding matrix product.

\begin{lemma}
Each $\phi\in\SL(2,\Z)$ is conjugate to $\pm \phi_w$ for some
word $w$ in the symbols $L,R$. 
The sign is unique and $w$ is determined up
to cyclic permutations of its letters. 
\end{lemma}

Let $M = M_\phi$ where $\phi =\pm \phi_w$ is hyperbolic. 
The so-called {\em monodromy triangulation} $T$ of $M$ has one tetrahedron
for each letter in $w$ and gluings 
determined by $w$ and
the sign. It is nicely described in \cite{FH} and  \cite{GF}. It follows from work of
Lackenby~\cite{La} that $T$ is the canonical ideal cell decomposition
of $M$. Other proofs of this result have recently been given by Gu\'eritaud (\cite{GF}, \cite{Guer}) and Akiyoshi, Sakuma, Wada and Yamashita (see \cite{ASWY}).

The intersection of $T$ with a (small) torus cross
section of the cusp of $M$, lifted to its universal cover $\R^2$, gives the {\em (lifted)
cusp triangulation} $T_0$ of $M$. Note that edges and vertices of
$T_0$ correspond to edges of $T$ seen transversely or end-on respectively. 
SnapPea~\cite{We} provides pictures of these cusp
triangulations: see Figure~\ref{lrrlr} for an example.

We need two things: the first is a combinatorial
description of $T_0$ in terms of $w$; the second is an understanding
of which edges and vertices of $T_0$ correspond to the same edges of
$T$ in $M$. Both are outlined briefly here:
for detailed explanations we refer the reader to the Appendix of \cite{FH}
and Sections 3 and 4 of \cite{GF}.

\subsection{The monodromy triangulation $T$}
The triangulation $T$ is built
up in layers by gluing tetrahedra according to the letters of
$w$. We begin with an almost flat ideal tetrahedron projecting onto a punctured torus;
the tetrahedron has edges $a,b,c,c_-$ identified as shown below.

\begin{center}
\includegraphics[scale=0.8]{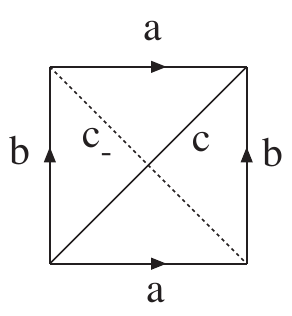}
\end{center}

For each successive letter $L$ or $R$ in the word $w$, 
we attach a tetrahedron to the top of the previous tetrahedron, as shown
below in the cover $(\R^2 - \Z^2) \times \R$.

\begin{center}
\includegraphics[scale=0.8]{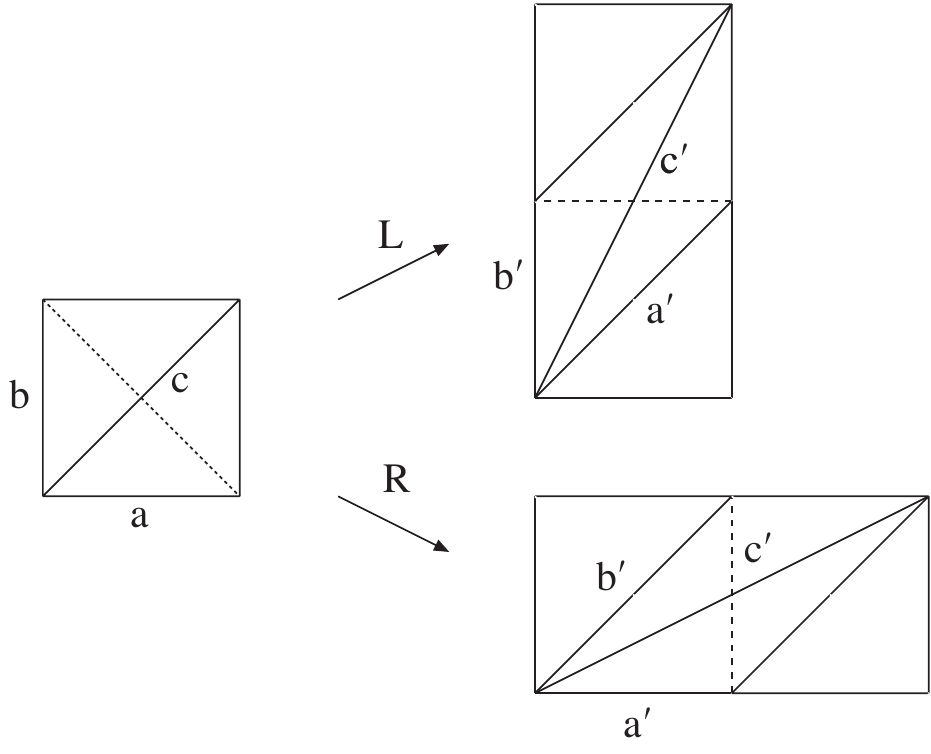}
\end{center}

After using all the letters of $w$, the final triangulation of the fibre $F$
differs from the initial triangulation by the monodromy  $\phi$,
and we can glue the top and bottom together to 
obtain an ideal triangulation $T$ of $M$.

\subsection{Combinatorial description of $T_0$.}
Now consider the induced triangulation of a cusp linking torus (i.e. cusp cross section) in $M$.
Each tetrahedron contributes a chain of 4 triangles 
going once around the cusp as shown in Figure \ref{cusp_cycle}.

\begin{figure}[h]
\begin{center}
\includegraphics[scale=0.8]{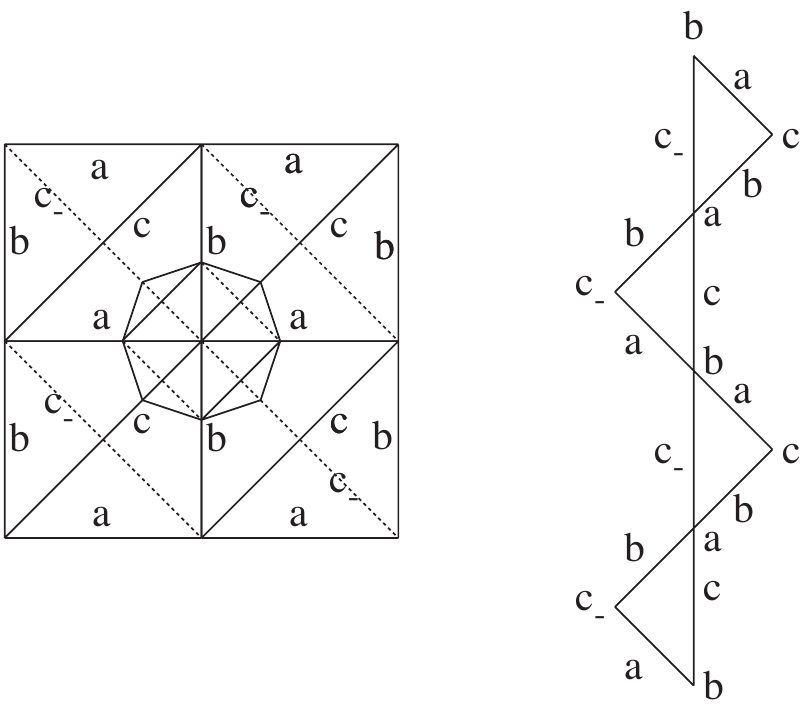}
\end{center}
\caption{The chain of triangles around a cusp coming from one tetrahedron.}
\label{cusp_cycle}
\end{figure}

In the triangulation $T_0$ of 
$\R^2$ this lifts to an infinite chain of triangles forming
a (vertical) saw-tooth pattern.

Each chain is glued to the next in one of
two ways depending on whether the 
letter is an $L$ or an $R$. 
\[\epsfig{file=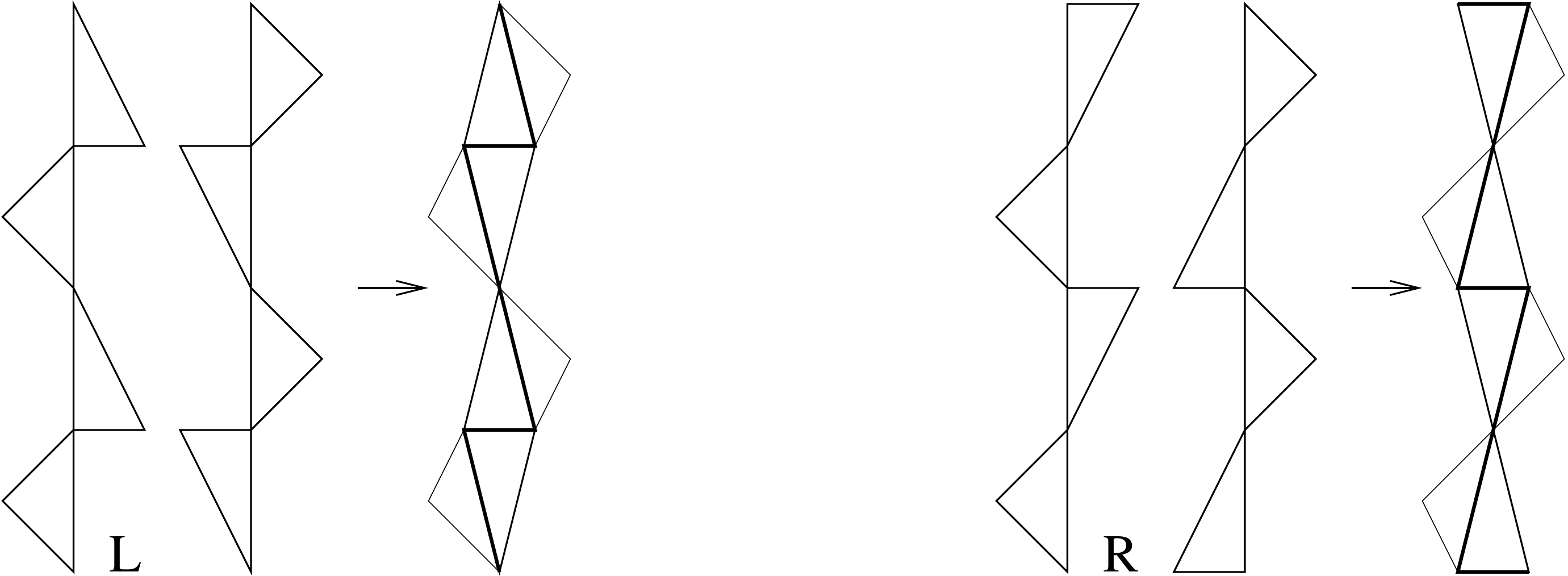,height=3.5cm} \]
Before gluing we adjust slightly the ``front'' triangles of the
first chain and the ``back'' triangles of the second, so as to create two
horizontal edges. 
After stacking these chains of triangles together we have a decomposition
of $T_0$ into  horizontal strips which can be described combinatorially
as follows.

Make a horizontal strip out of triangles corresponding to the
letters of $w$: for each $L$ add a triangle with a side on the bottom
edge of the strip and a vertex at the top; vice-versa for $R$; repeat
infinitely in both directions. For example, for $LRRLR$ we have: 
\[
\epsfig{file=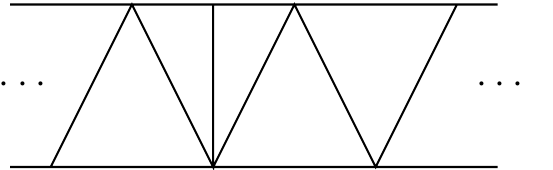}
\]
Fill the plane by reflecting repeatedly in the top and bottom 
edges of the strip. Then the cusp linking torus is the
quotient of the plane by the group $\G_0$ generated by 
 vertical translation by {\em four} strips, 
and horizontal translation by one period of the strip (composed with an extra vertical translation
 by  two strips if $\phi=-\phi_w$). 

\begin{figure}[h]
\epsfig{file=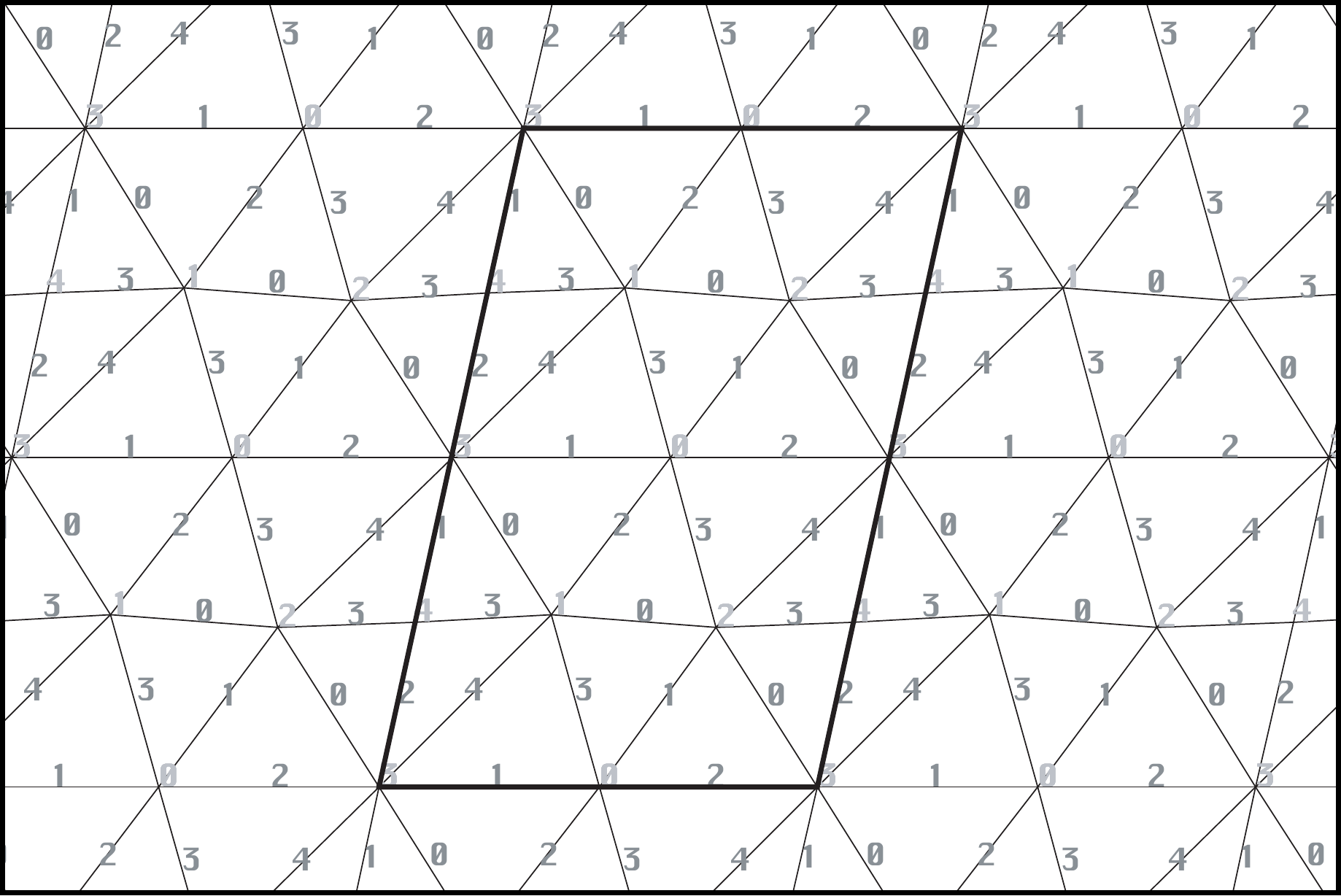,width=12cm}
\caption{The cusp triangulation of $M = M_{\phi_w}$ where $w = LRRLR$. The
triangulation of $M$ has five edges, labelled 0-4, which appear as both
edges and vertices of the cusp triangulation.}
\label{lrrlr}
\end{figure}

\subsection{Edge/vertex correspondence in $T_0$.}
To understand which edges of $T_0$ are identified in $T$,
we define a ``direction'' on each horizontal strip:
right to left on the first strip, left to right on the strips below
and above it, and so on, so that adjacent strips have opposite
directions. Then we have:

\begin{lemma}\label{edge_labels}
Each edge $e$ of $T_0$ which crosses a strip bounds two triangles, 
one of which, $\delta$ say, lies on the side of $e$ given by the direction on the
strip. Then $e$ and the opposite vertex of $\delta$ give the same
edge of $T$. Horizontal edges of $T_0$
correspond with the edges at the opposite
vertices of both adjacent triangles. 
\end{lemma}

This result is illustrated in Figure \ref{edge_idents} and will be 
important in the arguments below. 

\begin{figure}[h]
\begin{center}
\includegraphics[scale=0.8]{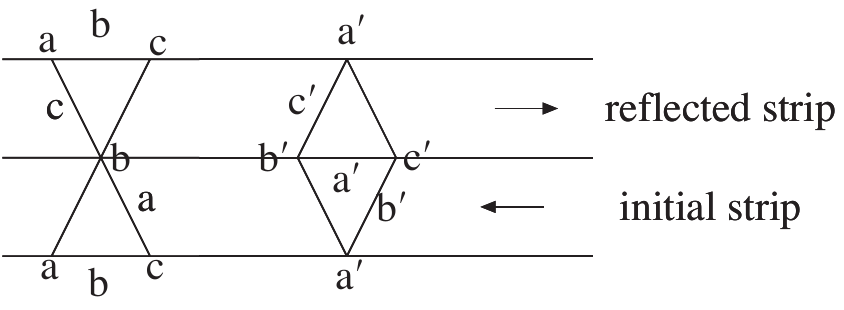}
\end{center}
\caption{Edge identifications in strips of $T_0$.}
\label{edge_idents}
\end{figure}

\begin{proof}
The boundary between two tetrahedra in $T$ is a punctured torus
consisting of two ideal triangles, homotopic to a fibre of $M$. In $T_0$ the
boundaries give {\em edge cycles,} lifts of a cycle of 6 edges
going once around the cusp in the vertical direction as in Figure \ref{cusp_cycle}. 
Here are two cycles: 
\[
\epsfig{file=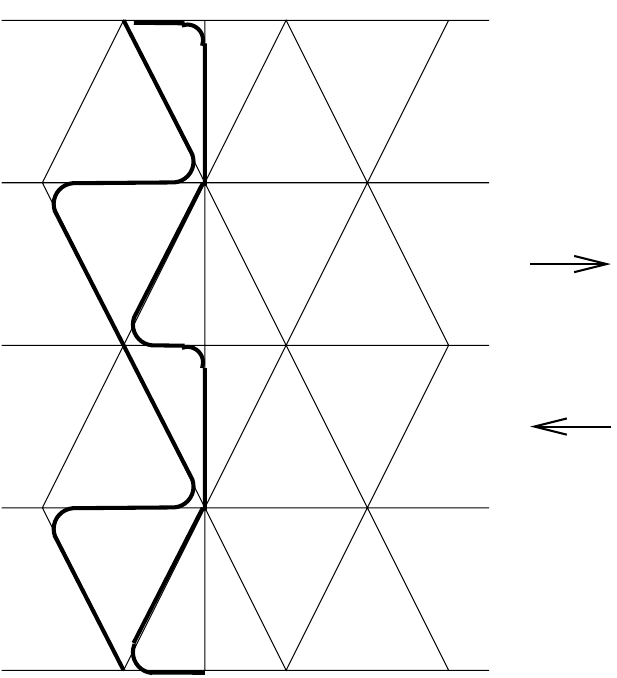,height=4cm}
\]
Cycles meeting a given strip form either ``$L$'s and $\Gamma$'s''
or ``backwards $L$'s and $\Gamma$'s.'' This gives the directions:
$L$'s and $\Gamma$'s read left to right!

We see from Figure \ref{cusp_cycle} that the
edges of a cycle are homotopic to 
three edges of $T$, $a,b,c$ say, cyclically
repeated. The vertices are the same three edges of $T$ arranged 
so that if an edge is $a$, its vertices are $b$ and $c$ and so on
(as in Figure \ref{edge_idents}).
Thus for any (forwards or backwards) $L$ or $\Gamma$:
the edge at the down-stroke (a strip-crossing edge) equals the edge seen at the other end of its horizontal stroke. 
Further, the edge seen at any horizontal stroke equals the edge at
the other end of the down-stroke. This proves the lemma.
\end{proof}

\subsection{Reduction to $T_0$.}
Let $\G$ be the image of the holonomy representation of $\pi_1(M)$ in
$\mbox{\rm Isom}(\H^3)$ so that $M = \H^3/\G$ as usual. 
Let $x$ be a fixed point of some maximal parabolic subgroup of
$\G$ which we identify with the group of translations $\G_0$ defined above. 
We regard symmetries and hidden symmetries as isometries of $\H^3$. 
Since $M$ is 1-cusped we can compose any symmetry or hidden symmetry
with an element of $\G$ to obtain an isometry which fixes $x$. 
Thus if $M$ has a symmetry, it is represented by a symmetry
of $T_0$; if it has a non-trivial hidden symmetry, it is represented 
by a symmetry of $T_0$ which does not come from a symmetry of $M$. 

Since our combinatorial picture of $T_0$ is not metrically exact, we only know that symmetries and
hidden symmetries act as simplicial homeomorphisms (S.H.) of $T_0$. 

\begin{lemma}\label{preserve_strips}
Simplicial homeomorphisms of $T_0$ preserve horizontal strips
whenever $w\neq (LR)^m$ or $(LLRR)^m$ as a cyclic word, for any $m>0$. 
\end{lemma}

Since vertex orders in $T_0$ are all even we can define a {\em
straight line} to be a path in the 1-skeleton which enters and leaves
each vertex along opposite edges. A {\em strip} consists of a part of
$T_0$ between two (infinite, disjoint) straight lines such that every
interior edge crosses from one side of the strip to the other. 

\begin{proof}
If an edge of $T_0$ lies on the edge of a strip, not necessarily
horizontal, the opposite vertex of the triangle which crosses that
strip must have order at least $6$: the straight line going through
this vertex has at least the two edges of the triangle on one side.

If there exists 
an S.H. which is not horizontal strip preserving, every vertex of $T_0$ 
will lie on a non-horizontal edge which is an edge of some strip
(namely the image of a previously horizontal strip edge). 

Suppose there is a vertex of order 10 or more, corresponding to 3 or
more adjacent $L$'s or $R$'s in $w$:
\[
\epsfig{file=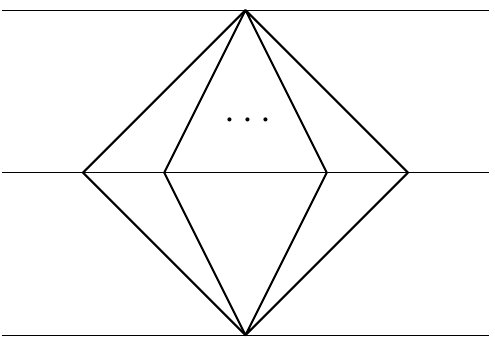, height=1.8cm}
\]
Here none of the non-horizontal edges 
shown can lie on a strip edge
because they are all opposite vertices of order 4. So in this case all
S.H.'s must preserve horizontal strips. 

Suppose there is a vertex of order 8, corresponding to $LL$ or $RR$ in
$w$: \[
\epsfig{file=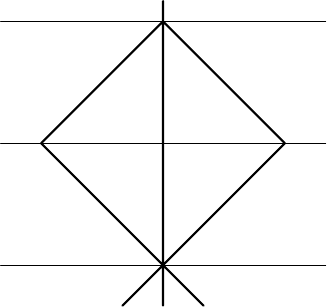, height=2cm}
\]
Now the only non-horizontal edges shown which can be strip edges are
the vertical ones. Vertex orders on this vertical edge alternate
$4,8,4,8$. So if an S.H. maps a horizontal strip edge to this vertical
one, vertex orders along some horizontal edge must be
$4,8,4,8,\ldots$. This determines $T_0$ and hence $w$ as $(LLRR)^m$. 
(The reader can verify that this particular $T_0$ admits a $\pi/2$
rotation.) 

Finally suppose there is no vertex of order $>6$. 
It follows
immediately that $w$ is $(LR)^m$ and $T_0$ is the tiling of the plane by
equilateral triangles. (Again, this admits S.H.'s which are not
horizontal strip preserving.) 
\end{proof}

\begin{lemma}\label{preserve_dir}
Simplicial homeomorphisms of $T_0$ coming from symmetries of $T$
preserve the horizontal strip directions whenever $w\neq (LR)^m$ or
$(LLRR)^m$. 
\end{lemma}

\begin{proof}
An S.H. of $T_0$ which comes from a
symmetry of $T$ obviously preserves the order of every edge of $T$. 
Thus if we label each edge of $T_0$ with the order of the
corresponding edge of $T$, the labels must be preserved. 

The order of an edge of $T$ corresponding to a vertex of $T_0$ is
simply its order as a vertex of $T_0$. By Lemma \ref{edge_labels} 
we can use the strip directions to
label the corresponding edges of $T_0$. Given a
vertex of order 10 or more we have: 
\[
\epsfig{file=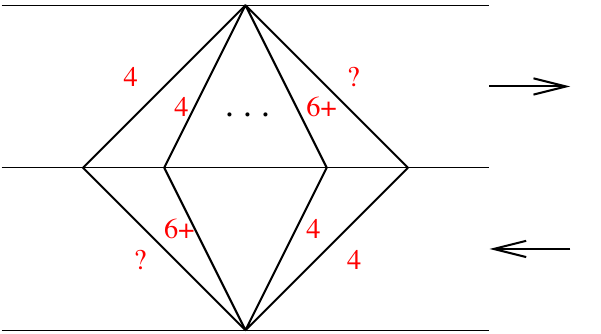, height=2.2cm}
\]
The direction on a horizontal strip cannot be reversed because this would swap an
edge of order 4 with one of order 6 or more. 

Suppose now the maximum vertex order is 8. For each order 8 vertex we
have: \[
\epsfig{file=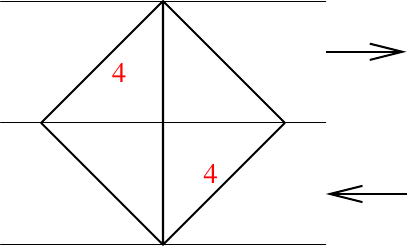,height=2cm}
\]
If there is an S.H. which reverses strip direction then the other two
sides of this diamond figure, wherever it appears, must also be
labelled with 4's.
\[
\epsfig{file=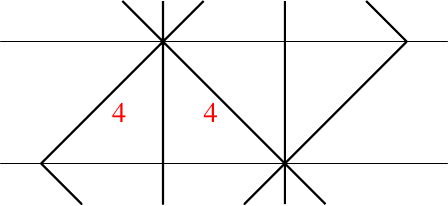,height=2cm}
\]
The 4 on the upper right edge implies that the right hand vertex of
the diamond has order $\geq 8$, hence $8$. This gives another diamond
figure adjacent to the first one. The argument can be repeated giving
the pattern associated with $w = (LLRR)^m$. 

If there are no vertices of order $>6$, $w = (LR)^m$. 
\end{proof}

\begin{proof}[Proof of Theorem \ref{nohs}]
It is shown in \cite{BMR} that the only arithmetic orientable hyperbolic
punctured torus bundles are those for which $w = LR, LLR$ (or $LRR$)
or $LLRR$, or powers of these, having invariant trace fields $\Q(\sqrt{-3}),
\Q(\sqrt{-7})$ and $\Q(\sqrt{-1})$ respectively. Thus, in the 
non-arithmetic cases, Lemmas \ref{preserve_strips} and \ref{preserve_dir}
show that all symmetries of $T_0$ which come from hidden symmetries of
$M$ are represented by simplicial homeomorphism which 
preserve the strips and strip-directions. 

Note that there are
two ``sister'' manifolds for each $w$ depending on whether the
monodromy is $\phi_w$ or $-\phi_w$; the triangulations of $\H^3$
however are the same, depending only on $w$. 

It is now easy to see that  the possible symmetries of $T_0$, 
modulo the translations in $G_0$, 
are restricted to the types listed below. We show that each,
if it occurs at all, comes from an actual symmetry of $M$.

\begin{enumerate}
\item Shifting up or down by two strips. This is realized by the
  symmetry ${\left(\begin{array}{rr}-1&0\\0&-1\end{array}\right)}\times
  1$ from $\F\times_{\pm\phi_w} S^1$ to itself. (Note that
  ${\left(\begin{array}{rr}-1&0\\0&-1\end{array}\right)}$ is central in
  $\SL(2,\Z)$.) \[
  \mbox{
  \begin{picture}(0,0)(0,0)
    \put(43,96){$\phi_w$} 
    \put(107,65){$\scalebox{.5}{$\minusmat $}\times 1$}
    \put(223,96){$\phi_w$} 
  \end{picture}\epsfig{file=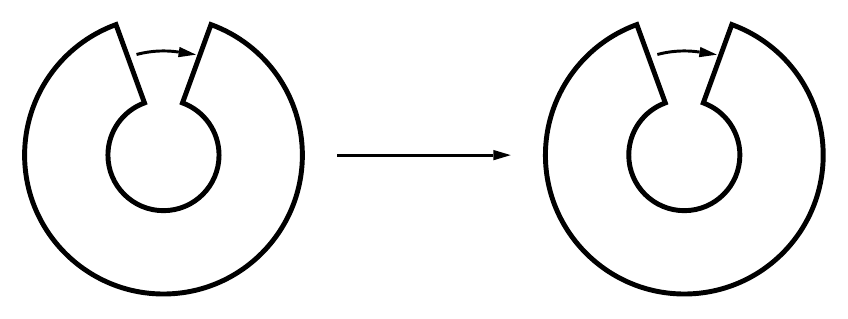,height=2.65cm}}
  \]

\item If $w$ is a power, $w = u^m$, then $T_0$ admits a horizontal
  translation by $\ell (u)$ triangles, where $\ell (u)$ is the length of the word $u$. 
  In this case 
\[
\begin{picture}(0,0)(0,0)
\put(42,20){$\phi_u$}
\put(100,54){$\phi_u$}
\put(157,20){$\phi_u$}
\end{picture}
M = \raisebox{-24pt}{\epsfig{file=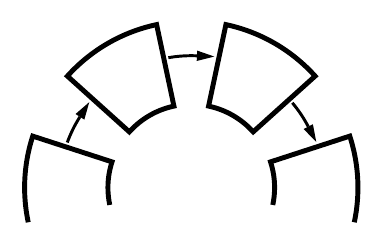,scale=0.95}} \]
and $1\times r_{2\pi/m}$ gives a symmetry of order $m$ (where
$r_\theta$ denotes rotation of the unit circle $S^1$ by $\theta$). For the $-\phi_w$
case we replace one of the $\phi_u$'s by $-\phi_u$. After rotating
we also have to apply $-1\times 1$ to $\F\times [0,1/m]$. 

\item If $w$ is {\em palindromic}, i.e. $w$ and its reverse $w'$ are
  the same as cyclic words, then $T_0$ admits a rotation by $\pi$
  about a point on one of the strip edges. This is realized by 
  $\pmmat \times \mbox{\em refl}(S^1)$, where $\mbox{\em refl}(S^1)$ 
denotes a reflection of $S^1$. \[
  \mbox{\begin{picture}(0,0)(0,0)
      \put(48,96){$\phi_w$}
      \put(110,60){$1\times\mbox{\em refl}(S^1)$}
      \put(227,96){$\phi_w^{-1}$}
      \put(295,65){$\scalebox{.5}{$\pmmat $}\times 1$}
      \put(408,96){$\phi_{w'}$}
    \end{picture}
  }\epsfig{file=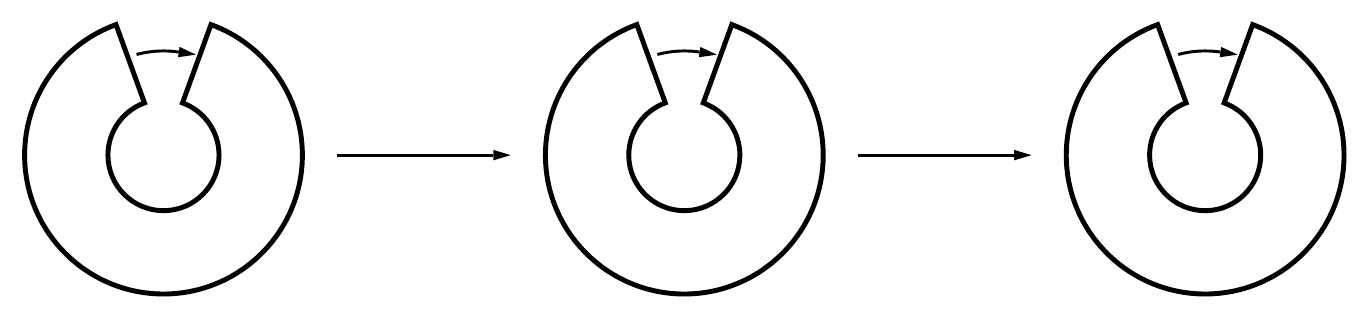,height=2.67cm} \]
  Note that conjugation by $\pmmat$ takes $L$ to $L^{-1}$ and $R$ to $R^{-1}$.

\item If $\ell(w)$ is even and rotating $w$ by a half turn swaps
  $L$'s and $R$'s, $T_0$ has a glide reflection mapping a strip to
  itself, exchanging the top and bottom of the strip. 
  This is realized by $\rflmat\times 1$ since conjugation by this
  matrix swaps $L$ and $R$. More explicitly, let $w = uv$ where $v$ is
  $u$ with $L$'s and $R$'s  interchanged.
  The symmetry is \[
   \mbox{\begin{picture}(0,0)(0,0)
      \put(62,96){$\phi_u$}
      \put(62,10){$\phi_v$}
      \put(150,68){$\scalebox{.5}{$\rflmat $}\times 1$}
      \put(282,96){$\phi_v$}
      \put(282,10){$\phi_u$}
    \end{picture}
  }\epsfig{file=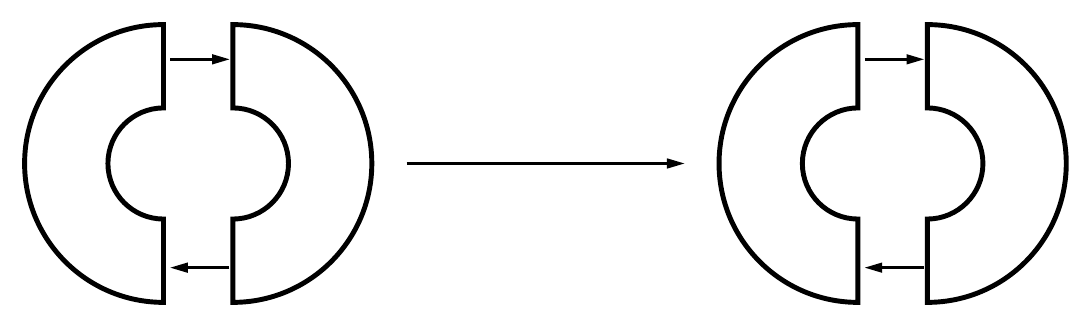,height=2.58cm} \]
  followed 
  by $r_{\pi}$ on the $S^1$ factor. This symmetry
  is orientation reversing. 

\item If $\ell(w)$ is even and reversing $w$ swaps $L$'s and $R$'s 
  (we might say that $w$ is {\em anti-palindromic}) then $T_0$
  admits a glide reflection with vertical axis, shifting everything up by one
  strip. This is realized by $\trnmat \times \mbox{\em refl}(S^1)$ 
  since conjugation by $\trnmat$ takes $L$ to $R^{-1}$ and $R$ to
  $L^{-1}$. This symmetry is orientation reversing. We leave the
  picture as an exercise for the reader. 
\end{enumerate}

\end{proof}

Clearly if $M$ admits any two of the symmetries 3--5 it admits the
third which is a product of the other two. Thus if $w$ is not a power,
$M$ may have one of 5 possible symmetry groups.

\section{Algorithm for finding isometries of tilings} \label{tilingisom}

To apply our commensurability criterion, Theorem \ref{comm_criterion}, we
need a way to determine when tilings of 
$\H^n$, arising from ideal cell decompositions, are isometric. 
We pursue this question in a slightly more general setting. 

Let $M$ and $M'$ be geometric $n$-manifolds with universal cover $X^n
= \E^n$ or $\H^n$, each with a given finite decomposition into convex
polyhedra. 
We wish to determine whether the tilings $T$ and $T'$ which
cover these two cell decompositions are isometric. (The elements of
$T$ (resp.\ $T'$) are convex polyhedra in $X^n$ which project to
polyhedra in the decomposition of $M$ (resp.\ $M'$).)

Necessary and sufficient conditions for the tilings to be isometric are as
follows.
\begin{enumerate}
\item There is an isometry $i$ of $X^n$ which maps an element of $T$
isometrically onto an element of $T'$.
\item
Whenever $i$ maps $P\in T$
isometrically onto $P'\in T'$, and $F$ is a (codimension $1$)
face of $P$, $i$ maps the neighbour of $P$ at $F$
isometrically onto the neighbour of $P'$ at $i(F)$.
\end{enumerate}
Sufficiency follows from the fact that we can proceed from any tile in
$T$ to any other by a finite sequence of steps between neighbouring
tiles.

Let $S$ and $S'$ denote the polyhedra decomposing $M$ and $M'$
respectively.

Let $\Theta$ denote the set of all triples $(j,p,p')$, where $p\in S$,
$p'\in S'$, and $j$ is an isometry carrying $p$ onto $p'$. We can find
$\Theta$ in a finite number of steps. Condition~1 above is
equivalent to $\Theta$ being non-empty.

We say that $(j,p,p')\in \Theta$ is {\em induced by} an isometry
$i$ of $X^n$ if there exist $P\in T$ projecting to $p$ and $P'\in T'$
projecting to $p'$ such that $i$ carries $P$ isometrically onto $P'$
and the restriction induces $j$. Let $f$ be a face of $p$ and let
$q$ and $q'$ be the neighbours of $p$ and $p'$ at $f$ and $j(f)$
respectively. The restriction of $j$ to $f$ induces an isometry
between certain faces of $q$ and $q'$. If this extends to
$(j_q,q,q')$ we say that $(j,p,p')$ {\em extends across $f$ to}
$(j_q,q,q')$. If $(j,p,p')$ is induced by $i$ then
$(j_q,q,q')$ will be also.

Condition~2 is equivalent to the following:
whenever $(j,p,p')$ is induced by $i$, and $f$ is a face of $p$, then
$(j,p,p')$ extends across $f$.

\begin{theorem}
With the above notation, the tilings $T$ and $T'$ are isometric if and
only if there exists a non-empty subset $I$ of $\Theta$ such that
every element of $I$ extends across each of its faces to yield another
element of $I$.
\end{theorem}

\begin{proof}
If $T$ and $T'$ are isometric with isometry $i$, simply let $I$ be the
set of elements of $\Theta$ induced by $i$.

Conversely, suppose we have $I\subseteq \Theta$ having the
stated properties.  Choose any element $(j,p,p')\in I$ and any
isometry $i$ of $X^n$ which induces it. Clearly condition~1 above is
satisfied.  For condition~2 note that if $i$ maps $P$ onto $P'$ and
this induces $(j,p,p')\in I$ then, because this extends across all of
its faces, $i$ maps neighbours of $P$ isometrically onto neighbours of
$P'$ and these too induce elements of $I$. Therefore condition~2 is
satisfied with all the induced triples belonging to $I$.
\end{proof}

An algorithm for finding such a subset $I\subseteq\Theta$ or
establishing that none exists is straightforward. If $\Theta$ is
empty, stop; there is no such subset. Otherwise choose any element of
$\Theta$ and put it in $I$. For each element of $I$ check if we can
extend across all faces. If we can't, remove all elements of $I$ from
$\Theta$ and start again. If we can, and every element we reach is
already in $I$, stop and output $I$. If we can extend across faces of
all elements of $I$, and we obtain new elements not yet in $I$, add
those new elements to $I$ and check again.

Each set $I$ found by the algorithm represents an equivalence class of
isometries carrying $T$ onto $T'$: choose an isometry $i$ inducing any element of
$I$; then isometries $i,i'$ are {\em equivalent} if $i' =g'ig$ where 
$g$ is a covering transformation of $M$ and 
$g'$ is a covering transformation of $M'$.

To find the symmetry group of a tiling $T$ we apply the algorithm with
$T=T'$, $M = M'$ etc. Let $\Gamma$ denote the group of covering
transformations of $M$.  Repeating the algorithm until $\Theta$ is
exhausted we obtain finitely many equivalence classes of symmetries.
These represent the double cosets of $\Gamma$ in $\mbox{\rm Symm}(T)$.
Together with $\Gamma$ they generate $\mbox{\rm Symm}(T)$.

\begin{remark}\label{iscovering}
Each set $I$ found by this algorithm also represents a common covering $N$ of
$M$ and $M'$ constructed as follows: 
For each element $r = (j,p,p')\in I$ let $p_r$ denote a
copy of $p$.  Let $N$ be the disjoint union of the polyhedra $p_r$ as
$r$ varies over $I$, with the following face identifications. Each face $f$ of $p$
appears as a face $f_r$ of $p_r$.  Whenever $r$ extends across
$f$ to $s = (j_q,q,q')$,
identify face $f_r$ of $p_r$ with face $f_s$ of $q_s$.
\end{remark}

\subsection{Example}
Let $M$ be the Euclidean torus obtained by gluing a $1$ by $3$
rectangle along opposite pairs of edges. Let $M'$ be another such
torus obtained by gluing a $2$ by $1$ rectangle. Subdivide each into
unit squares as shown, so that $S$, the subdivision of $M$, equals
$\{A,B,C\}$ while $S' = \{P,Q\}$. Both $S$ and $S'$ lift to tilings
of the plane by unit squares.
$$\epsfig{file=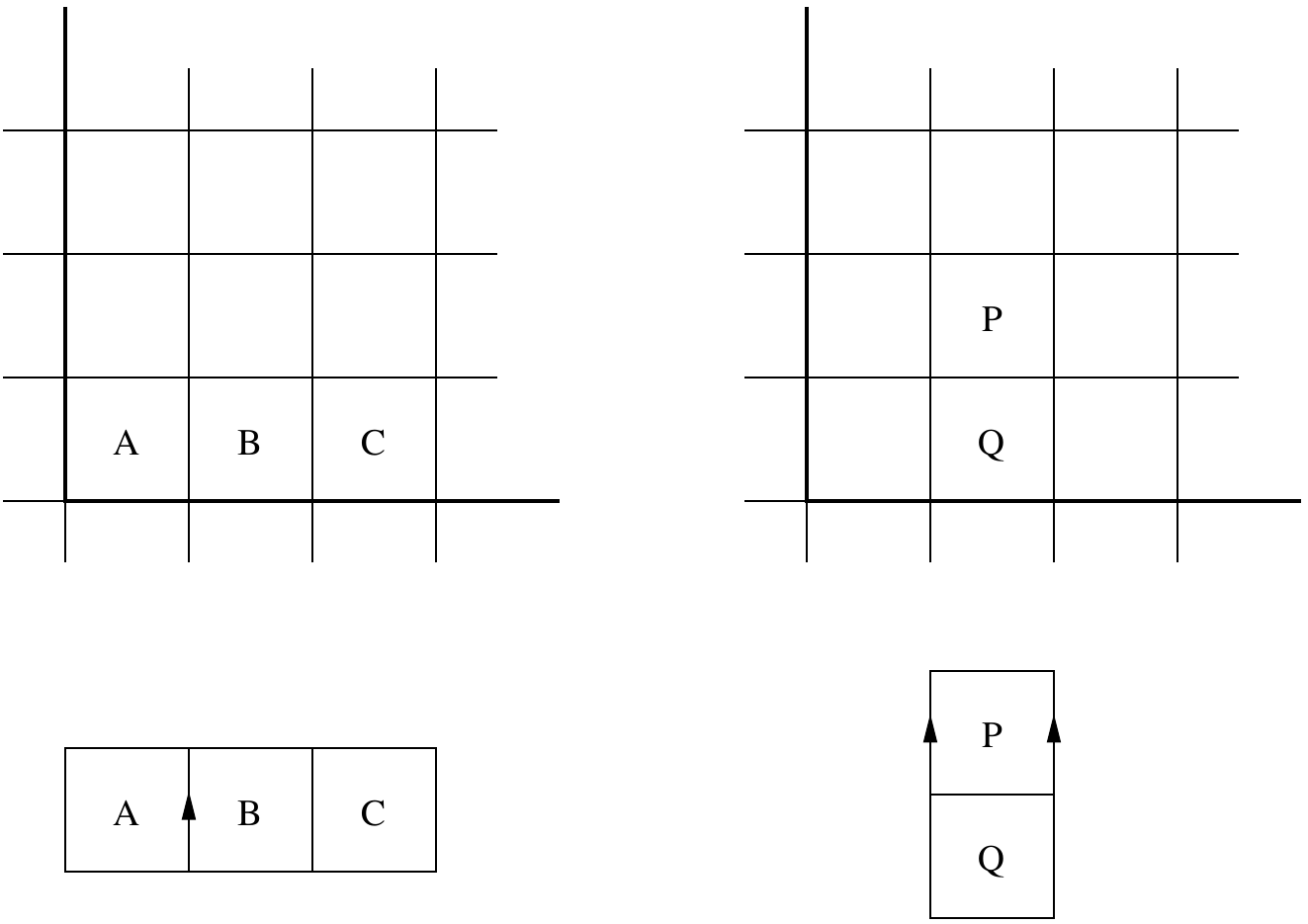,height=5cm}$$
$\Theta$ consists of $48$ elements since there are $8$ ways of isometrically 
mapping
one unit square onto another and $6$ combinations of squares to be
mapped. Suppose we start with $(t_{AP},A,P)$ where $t_{AP}$ denotes
the translation carrying $A$ onto $P$ in the picture. We can extend
$(t_{AP},A,P)$ across the marked edge (whose image in $M'$ is also
shown) to obtain $(t_{BP},B,P)$. Continuing until we have extended
across every available edge we obtain $6$ elements of
$\Theta$ which we can abbreviate to $\{t_{AP},t_{BP},t_{CP},
t_{AQ},t_{BQ},t_{CQ}\}$. These give rise to the following common
covering of $M$ and $M'$.
$$\epsfig{file=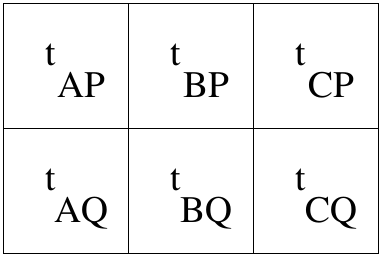,height=1.35cm}$$
Suppose instead we begin with a clockwise rotation $r_{AP}$ through
$90^{\circ}$ carrying $A$ onto $P$. Then we obtain a different common
cover.
$$\raisebox{0cm}[1cm]{\epsfig{file=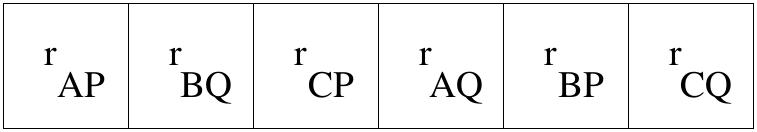,height=.725cm}}$$
Suppose finally we subdivide one of the squares in $S'$ into two
triangles. Now the tilings are clearly not isometric.
$$\epsfig{file=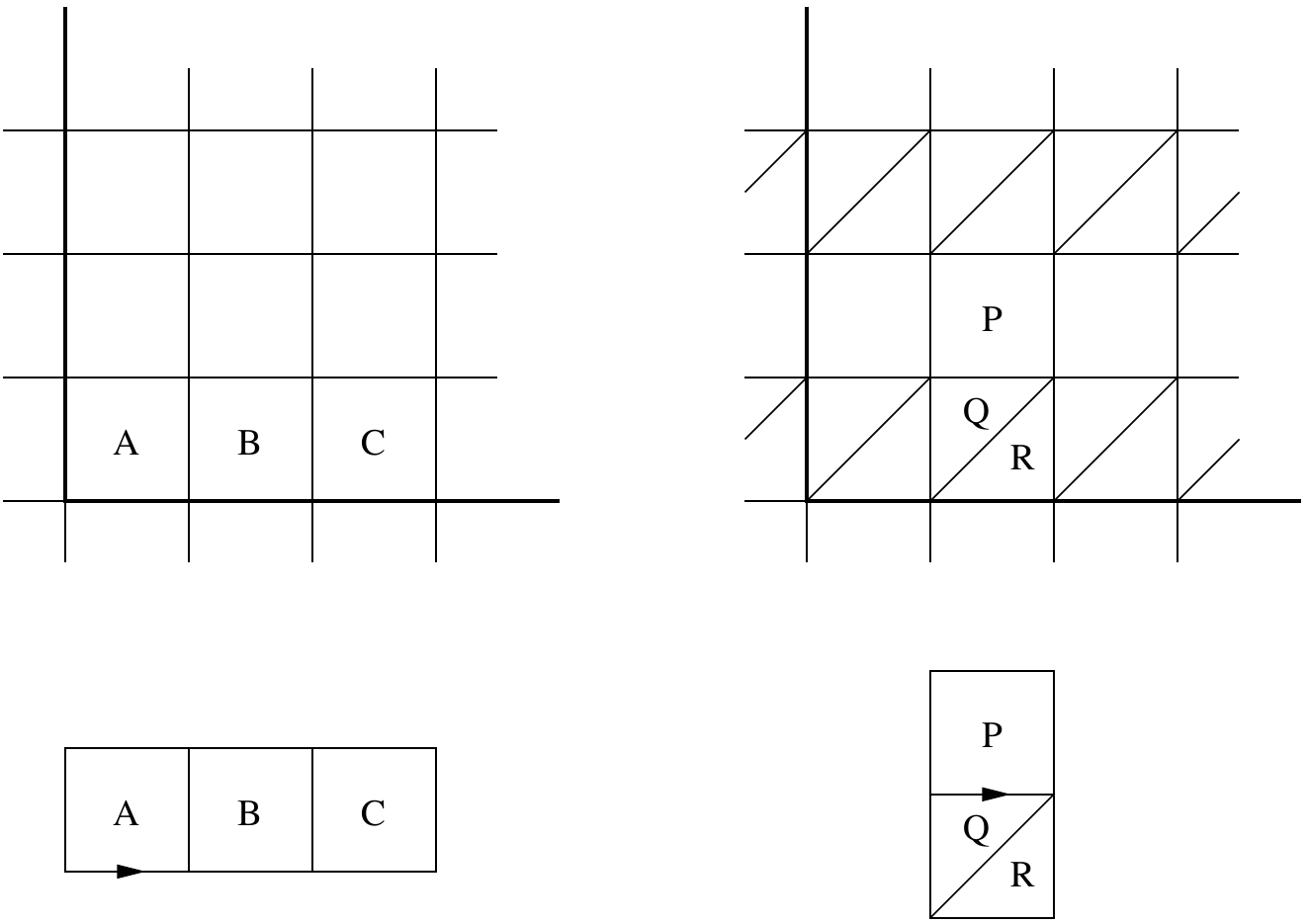,height=5cm}$$
The algorithm
might start with $t_{AP}$ but when it tries to extend this across the
marked edge it will find that the induced mapping from an edge of $A$
to an edge of $Q$ does not extend to any isometry from $A$ to $Q$.

\subsection{Combinatorial construction of coverings}\label{combisom}
Mostow-Prasad rigidity often allows us to work in a purely combinatorial setting;
this motivates the following. 
The algorithm described above can be viewed as a method for
constructing a common cover of two manifolds which admit suitable
subdivisions into polyhedra. 
The algorithm works equally well for topological spaces
$M$ and $M'$ admitting $n$-dimensional polyhedral decompositions 
(i.e. obtained by gluing $n$-dimensional, Euclidean or
finite volume hyperbolic, convex polyhedra pairwise along their
$(n-1)$-dimensional faces, such that open faces of dimension $q\leq n$
embed).  We then let 
each element $(j,p,p')\in \Theta$ represent a
combinatorial isomorphism between $p$ and $p'$.  (We can barycentrically subdivide
the decompositions of $M$ and $M'$ in order to realize the $j$'s as
piecewise linear homeomorphisms in a canonical way.) Now, if the
algorithm is successful it constructs, in general, a branched covering
of $M$ and $M'$, branched over the codimension $2$ skeleta of $M$ and
$M'$.

When $M$ and $M'$ are PL-manifolds, we would like to construct an
unbranched cover
which is also a manifold. It is sufficient to check that the cover we
construct does not branch over any codimension $2$ cell of $M$ or
$M'$. For if this is the case we proceed by induction on $q$ to show
that our cover is not branched over any codimension $q$ cell of $M$
(or $M'$). 
By construction, there is no branching over cells of codimension $0$ and $1$,
and we assume that there is no branching over cells of codimension 2.
So suppose $q \ge 3$ and there is no branching over  cells of codimension $\leq q-1$. 
Then the
link of a codimension $q$ cell in the cover is a connected unbranched
covering of the link of a codimension $q$ cell of $M$. But since the
latter is a $(q-1)$-sphere, hence simply connected since $q\ge3$, the covering is a
homeomorphism.

In the case of most interest to us, namely when $M$ and $M'$ are
cusped hyperbolic $3$-manifolds, Mostow-Prasad rigidity ensures that if $N$
topologically covers both $M$ and $M'$ then the induced hyperbolic
structures on $N$ are isometric.

Frequently, an ideal cell decomposition of a hyperbolic
$3$-manifold consists entirely of ideal tetrahedra. If this is the
case for $M$ and $M'$, then our algorithm will always find a common
cover, branching over the edge sets of the two manifolds 
(elements of $\Theta$ will always extend over their faces because 
all tetrahedra are combinatorially equivalent). 
This will indicate that the two manifolds are commensurable only if the covering
is in fact unbranched. When an element of $\Theta$ induces a mapping
of an edge $e$ of (the ideal cell decomposition of) $M$ onto an edge
$e'$ of $M'$, the resulting cover will be unbranched if and only if the
order of $e$ (the number of tetrahedra to which glue around it) is
equal to the order of $e'$. For example, this has the following corollary.

\begin{theorem} If a
finite volume hyperbolic $3$-manifold $M$ admits a
decomposition into ideal tetrahedra such that the order of every edge
is $6$, then $M$ is commensurable with the complement of the
figure-$8$ knot.
\end{theorem}

\section{Enumerating canonical cell decompositions I}
For 1-cusped hyperbolic manifolds with discrete commensurator we now
have all the ingredients of an effective method for testing
commensurability. For multi-cusped manifolds we need a way to search
through the (finite) set of all canonical ideal cell decompositions. 
In this section and the next we describe two alternative approaches to
this problem, while in Section~\ref{cuspcom} we show how the search can 
be restricted for greater efficiency. 

The first approach is very simple minded: if we can
bound the degree with which $M$ covers its commensurator quotient, we can
enumerate all cusp cross sections which could possibly cover equal
area cross sections in the quotient. Such degree bounds can be
obtained from estimates on the the minimum volume of cusped non-arithmetic
hyperbolic 3-manifolds (\cite{Me}, \cite{Ad1}, \cite{NR1}).

The second approach is geometric and may be of more theoretical
interest since it truly finds
{\em all} the canonical cell decompositions. In fact it associates
with a given $c$-cusped $M$, a convex polytope in $\R^c$ whose
$k$-dimensional faces, for $0\leq k < c$, are in 1-1 correspondence
with canonical cell decompositions of $M$. 
\medskip

Let $M$ be a hyperbolic orbifold with (horoball) cusp neighbourhoods
 $C_1,\ldots,C_m$ and let
$N$ be a degree $d$ quotient of $M$ with corresponding cusp
neighbourhoods $c_1,\ldots,c_n$. Let
$\pi:M \rightarrow N$ be the covering projection and let 
$c_{j(i)} = \pi(C_i)$.  If we choose horospherical cross sections in $N$ with ``area''
(i.e. codimension 1 volume) equal to  $1$,
the area of $\partial C_i$ is equal to the degree with which $C_i$ covers
$c_{j(i)}$. The sum of the areas of the $C_i$ covering $c_j$ will be
$d$ for each $c_j$.

Thus, in order to find a possible degree $d$ quotient of $M$ via the
canonical cell decomposition, we can enumerate the possible integer
area vectors as follows. For each $n$, $1\leq n\leq m$, and each
partition of $\{1,\ldots,m\}$ into $n$ non-empty subsets $I_1,\ldots,I_n$,
enumerate all area vectors $(a_1,\ldots,a_m)$ such that each $a_i$ is a positive integer and $\sum_{i\in I_j} a_i = d$ for $j=1,\ldots,n$.

Example: if $m=3$ we have partitions $\{\{1,2,3\}\}$, $\{\{1,2\},\{3\}\}$,
$\{\{1,3\},\{2\}\}$, $\{\{1\},\{2,3\}\}$ and $\{\{1\},\{2\},\{3\}\}$. The
first partition admits $\frac{1}{2}(d-1)(d-2)$ area vectors, the next
three admit $d-1$ each, and the last, just one. 

We can enumerate area vectors corresponding to possible quotients of degree $d$ as follows.
Any quotient of $M$ of degree $d$ has $n \le m$ cusps and has a corresponding area vector
$(a_1, \ldots, a_m)$ with sum $\sum_{i=1}^m a_i = nd$. 
So if we fix an integer $D \ge m$ and enumerate all positive
integer area vectors summing to at most $D$, we will find all
quotients of degree $d$ with $n\le m$ cusps such that $nd \leq D$. 
In particular this will include all quotients with at most $m-1$ cusps provided
$d \le D/(m-1)$.  Since the canonical cell decomposition is determined by the {\em ratio}
of the areas, any $m$ cusped quotient is found using the area vector $(1,\ldots,1)$
with $\sum_i a_i = m \le D$. So we will find {\em all} canonical cell decompositions arising from quotients of degree $d \leq D/(m-1)$.

\section{Enumerating canonical cell decompositions II}

Let $M$ be a hyperbolic $n$-manifold with $c>0$ cusps.  Then the set of
possible choices of (not necessarily disjoint) horospherical cross sections dual
to the cusps of $M$ is parametrized by the vector of their areas 
(i.e. $(n-1)$-dimensional volumes)
 in $\R_{>0}^c = \{(v_1,\ldots,v_c)\in \R^c \mid v_1>0, \ldots, v_c > 0 \}.$ 
 In fact it turns out to be more convenient to parametrize
by the $(n-1)$th root of area, a quantity we will refer to as {\em size.}
Multiplying a size vector $v\in \R_{>0}^c$ by a constant $\lambda>0$
has the effect of shifting the corresponding horospherical cross sections a
distance $\log(\lambda)$ down (i.e. away from) the cusps.

For each $v\in\R_{>0}^c$ we obtain a Ford spine.  It is clear from the
definition that if we choose a set of disjoint horospherical  cross sections and
then shift them all up or down by the same amount we get the same Ford
spine. Thus the set of possible Ford spines is parametrized by the set
of rays in $\R_{>0}^c$, or equivalently, by points in the open
$(c-1)$-dimensional simplex ${\mathcal S} = \{(v_1,\ldots,v_c)\in
\R_{>0}^c \mid v_1+\dots +v_c = 1\}.$ For a $1$-cusped manifold the
Ford spine is unique. 

Dual to each Ford spine is a {\em canonical cell decomposition} $D(v)$. As
we vary $v\in\mathcal S$ the decomposition changes only when the
combinatorics of the spine changes. To better understand this
dependence we now review an
alternative approach to defining $D(v)$, namely the original
one of Epstein-Penner in \cite{EP}.

We work in Minkowski space $\E^{n,1}$ with the inner product 
$\ast$ defined by
$$x \ast y = x_1 y_1 + \ldots + x_n y_n - x_{n+1} y_{n+1}$$
for $x=(x_1, \ldots, x_{n+1})$, $y=(y_1, \ldots, y_{n+1})$ in $\R^{n+1}$.
Then hyperbolic space $\H^n$ is the upper sheet of the hyperboloid $x*x=-1$,
and the horospheres in $\H^n$ are represented by the intersections of hyperplanes, having light-like normal vectors, with $\H^n$. Each such hyperplane $H$ has a unique Minkowski normal $\mathbf n$ such that $x\in H$ if and only if $x\ast{\mathbf n} = -1$.

Let $\Gamma$ denote the group of covering transformations of $\H^n$
over $M$.  Each size vector $v$ gives rise to a $\Gamma$-invariant set
of horospheres in $\H^n$. The resulting set of normals in Minkowski space 
is invariant under the action of the group $\Gamma$. 
The convex hull of this set of points, which we shall refer to as
the {\em Epstein-Penner convex hull,} intersects every ray based
at the origin passing through a point in the upper sheet of the hyperboloid.
The boundary of this convex set is a union of closed convex
$n$-dimensional polytopes having coplanar light-like vertices.
Epstein and Penner \cite{EP} show that  
these project to a locally finite, $\Gamma$-invariant set of ideal polyhedra
in $\H^n$, which in turn project to a finite set $D(v)$ of ideal hyperbolic
polyhedra in $M$.

Starting with a given ideal hyperbolic cell decomposition $D$ of $M$ and a
size vector $v$, we proceed next to describe necessary and sufficient
conditions for $D$ to be the canonical cell decomposition $D(v)$, as in
\cite{We1} and \cite{SW1}.

Let $C$ be a cell of $D$. Then $v$ determines a horospherical
cross section to each ideal vertex of $C$. 
We lift $C$ and this choice of horospheres to $\H^n$. In Minkowski space, this gives a convex (Euclidean) $n$-dimensional polytope whose vertices are the hyperplane normals for these horospheres. 
Whenever $C$ is not a simplex, it is necessary to add the condition that these vertices are coplanar. If this is satisfied for all the non-simplicial
cells of $D$, we can lift each cell to an $n$-dimensional polytope in
Minkowski space with vertices corresponding to the choice of horospheres determined by the size vector $v$. If $C$ and
$C'$ are neighbouring cells in $D$ it is necessary that the angle
between neighbouring lifts into Minkowski space be convex upwards. Together
these conditions are also sufficient to imply $D = D(v)$.

These conditions can be expressed as a set of linear equations and linear inequalities on the entries of the size vector $v$.

\begin{proposition} \label{conditions_for_D(v)}
Let $D$ be an ideal hyperbolic cell decomposition of a cusped hyperbolic $n$-manifold $M$ with $c$ cusps.
Then there exist matrices $L_D$ and $F_D$ with $c$ columns such that, for 
$v\in\R_{>0}^c$,
$D(v) = D$ if and only if $L_D v=0$ and $F_D v>0$.
\end{proposition}

\noindent Note: here $v$ is written as a column vector, and the condition $F_D v>0$ means that each entry of the vector $F_D v$ is positive.

\begin{proof}
Let $v_i$ denote the entry of $v$ corresponding to
the $i$th cusp of $M$. Let $\n_j$ be the vertex representative for the
$j$th vertex of $C$ lifted to Minkowski space for the choice of horospheres 
given by $v = (1,\ldots ,1)$. Then for an arbitrary size vector $v$
the corresponding representative is $\n_j/v_{c(j)}$, where $c(j)$ is the
cusp of the $j$th vertex of $C$. 

The coplanarity condition on a non-simplicial cell $C$ gives a set of
linear equations satisfied by $v$, one for each vertex of $C$ in
excess of $n+1$, as follows.  Let $N_v$ be a Euclidean normal to the
hyperplane containing $\{\n_0/v_{c(0)},\ldots ,\n_{n}/v_{c(n)}\}$ such
that $(\n_j/v_{c(j)})\cdot N_v=1$ for $j=0, \ldots, n$, 
where $\cdot$ denotes the Euclidean dot product. 
Writing $M_C$ for the inverse of the
matrix with rows $\n_j$ we obtain $N_v = M_C (v_{c(0)},\ldots, v_{c(n)})^t$,
which is a linear function of $v$. For $j>n$, $\n_j/v_{c(j)}$ belongs
to this hyperplane if and only if $\n_j\cdot N_v - v_{c(j)} = 0$, which is
linear in $v$.  The full set of constraints for $C$ gives 
a matrix equation $L_C v=0$.

The convexity condition at an $(n-1)$-cell $f$ of $D$, being the
common face of $n$-cells $C$ and $C'$, can be expressed as follows.
Let $N_v$, as above, be the defining normal for the hyperplane
containing the lift of $C$ determined by $v$. Let $\n'_k/v_{c(k)}$ be
a vertex of an adjacent lift of $C'$, not in the lift of $f$. This
vertex lies above the hyperplane if and only if $(\n'_k/v_{c(k)})\cdot N_v
> 1$, or equivalently $\n'_k\cdot N_v - v_{c(k)} > 0$.  We refer to the
left-hand side of this inequality as the {\em tilt} at $f$ of $v$ and
express the condition as $F_f v > 0$, where $F_f$ is a suitable
row-vector. (Note that the sign of our tilt function is opposite to
that of \cite{We1} and \cite{SW2}.)

Finally, concatenate the matrices $L_C$ into a matrix $L_D$
and the rows $F_f$ into a matrix $F_D$. 
\end{proof}

Let $\P_D$ denote the set of $v\in\R_{>0}^c$ such that $L_D v=0$ and
$F_D v>0$. We call this the {\em parameter cell} of $D$ since it
contains all cusp size parameters $v$ such that $D(v)=D$. Each
$v\in\R_{>0}^c$ belongs to a parameter cell, namely $\P(v) =
\P_{D(v)}$. The parameter cell $\P_D$ is non-empty if and only if $D$ is a canonical cell
decomposition.

It is shown in \cite{Ak} that
the number of canonical cell decompositions is finite. Therefore
$\R_{>0}^c$ is a union of finitely many parameter cells.

\begin{proposition} \label{perturb}
Each $v\in \R_{>0}^c$ can be perturbed to obtain a nearby vector
$v'$ such that $\P(v')$ has dimension $c$ and $\P(v)$ is a face of
$\P(v')$ (or equals $\P(v)$ if this has dimension $c$).
\end{proposition}

\begin{proof}
Let $D = D(v)$.
If $L_D$ is zero (or empty) then $\P(v)$ is an open subset, hence a $c$-dimensional
cell and we just set $v'=v$.

Otherwise, perturb $v$ such that it leaves the linear subspace
determined by $L_D v=0$. For a small perturbation the Epstein-Penner
convex hull changes as follows: no dihedral angle between adjacent
$n$-faces goes to $\pi$ but some non-simplicial $n$-faces may be
subdivided if their vertices become non-coplanar. It follows that
$L_D$ may lose rows and $F_D$ may gain rows. Let $D'$ be the new
decomposition. Then $L_{D'}\neq L_D$ because $L_{D'} v'=0$ while
$L_D v'\neq 0$. Repeat until $L_{D'}$ is zero (or empty). Then $\P(v')$ has 
dimension $c$.

Now $v$ belongs to a face of $\P_{D'}$ since it satisfies $F_f v > 0$
for each face $f$ common to $D$ and $D'$, and $F_{f'} v = 0$
for each face $f'$ of $D'$ not in $D$.  The former condition amounts
to $F_D v > 0$. We have to show that the latter is equivalent to
$L_D v=0$. But that is equivalent to the coplanarity of the lifted
vertices of each $n$-cell of $D$. Such a cell may be subdivided by new
faces $f'$ in $D'$. Then the vertices will be coplanar at $v$ if and
only if the tilt at each subdividing face is zero, i.e. if and only if
$F_{f'} v = 0$ for the subdividing faces.
\end{proof}

The above proposition implies that each face of a parameter
cell is another parameter cell; the decomposition corresponding to a
parameter cell is a refinement of the decompositions corresponding to
its faces.

\begin{remark} \label{non-generic-cell}
It is tempting to suppose that the canonical cell decomposition of $M$
corresponding to a $c$ dimensional parameter cell must consist entirely
of ideal simplices but this need not be the case. In general we can
have cell decompositions with non-simplicial cells such that $L_D$ is
a zero matrix. (For example, this occurs for the Borromean rings complement --- see Section \ref{b_rings_ex}  below.)
\end{remark}

We now have the following algorithm for finding all canonical cell
decompositions. First we find all the $c$-dimensional
parameter cells.
\begin{enumerate}
\item Choose an arbitrary $v\in \R^c_{>0}$.
\item Perturb $v$ if necessary,
as in the proof of Proposition~\ref{perturb}, so that $\P(v)$ has
dimension $c$, and add it to our list of cells.
\item If the closure of the cells we have found so far does not
contain the whole of $\R^c_{>0}$, choose a new $v$ not in the closure of
any cell found so far and repeat step 2.
\end{enumerate}
By the finiteness result quoted above, this algorithm eventually
terminates. We can then enumerate all canonical cell decompositions by
enumerating the faces of all dimensions of the cells $\P(v)$.

\medskip

While the computational geometry involved in implementing the above
algorithm is certainly possible, it is not particularly nice. We
explain a refinement which gives a little more insight and an algorithm which
is easier to implement.

For a decomposition $D$ of $M$, let ${\Sigma}_D$ denote the row
vector obtained by adding together the rows of $F_D$. We define the
{\em tilt polytope} of $M$ to be the set of $v\in\R^c_{>0}$ such that
${\Sigma}_D\cdot v < 1$ for all canonical cell decompositions $D$ of
$M$.

\begin{proposition}
The tilt polytope $T$ of $M$ is bounded. The parameter cells $\P_D$ of $M$
are the cones over the origin of those faces of $T$ which are not
contained in $\partial \R^c_{>0}$.
\end{proposition}

\begin{proof}
We show that the closure of a $c$-dimensional parameter cell
$\overline{\P}_D$ has bounded intersection with $T$. Let $v$ be a unit
vector in $\overline{\P}_D$. Then since $v$ is not contained in every
face of $\P_D$, ${\Sigma}_D\cdot v > 0$.  The length of any multiple of
$v$ contained in $T$ is bounded by $1/({\Sigma}_D\cdot v)$.  Since this is
continuous in $v$, and the set of such $v$ is compact,
$\overline{\P}_D \cap T$ is bounded. Since $T$ is a union of finitely
many such sets it is bounded.

Let us write $H_D$ for the half-space $\{x\in\R^c\mid {\Sigma}_D\cdot x <
1\}$. Then $T$ is the intersection of all the $H_D$'s with
$\R^c_{>0}$. We will show that: {\em if $v$ belongs to a $c$-dimensional
parameter cell $\P_D$, and $\P_{D'}$ is any other
parameter cell, then the ray generated by $v$ leaves $H_D$ before it
leaves $H_{D'}$}. 

It will then follow that a ray in $\P_D$ penetrates the (non-empty) face
of $T$ generated by $H_D$. Since a ray not in $\P_D$ belongs to
the closure of some other parameter cell, it does not leave $H_D$
first and therefore does not pass through the same face of $T$. Since
cones on the lower dimensional faces of a $(c-1)$-dimensional face of
$T$ are the faces of a $c$-dimensional parameter cell, the result
then follows from Proposition~\ref{perturb}.

It remains to show that $H_D$ cuts off any ray in $\P_D$ closer to the
origin than
$H_{D'}$, for all parameter cells $\P_{D'}\neq \P_D$. 
Equivalently, for $v\in \P_D$, ${\Sigma}_D\cdot v > {\Sigma}_{D'}\cdot v$. 

Firstly, let $\P_D$ and $\P_{D'}$ be any two parameter cells such that
$\P_{D'}$ is a face of $\P_{D}$, and let $v$ belong to $\P_D$. The rows of
$F_{D'}$ are a proper subset of the rows of $F_D$, and since $F_f\cdot v>0$ for
each row, ${\Sigma}_D\cdot v > {\Sigma}_{D'}\cdot v$.
If instead $\P_{D}$ is a face of $\P_{D'}$, then the rows of
$F_{D'}$ omitted from $F_D$ are precisely those for which $F_f\cdot v$
vanishes. Therefore in that case ${\Sigma}_D\cdot v = {\Sigma}_{D'}\cdot v$.

Next, let $\P_D$ be a $c$-dimensional parameter cell, and let $\P_{D'}$
be arbitrary, with $D'\neq D$. Choose $v\in \P_D$ and $v'\in \P_{D'}$.
Let $\P_{D_1},\ldots,\P_{D_m}$ be the parameter cells through which
the straight line $v_t:= (1-t)v+tv'$ passes for $0\leq t \leq 1$ (so
that $D_1=D$ and $D_m=D'$).  For $v_t$ in $\P_{D_i}$,
${\Sigma}_{D_{i}}\cdot v_t \geq {\Sigma}_{D_{i+1}}\cdot v_t$ while for $v_t\in
\P_{D_{i+1}}$, ${\Sigma}_{D_{i}}\cdot v_t \leq {\Sigma}_{D_{i+1}}\cdot v_t$.
Since the difference between these terms is (affine) linear in $t$,
the former inequality must hold for all lesser values of $t$, in
particular, when $v_t=v$. Note also that the first such inequality is
strict, namely, ${\Sigma}_{D_1}\cdot v > {\Sigma}_{D_2}\cdot v$.  It follows
that ${\Sigma}_{D}\cdot v > {\Sigma}_{D'}\cdot v$ for arbitrary $D'$.
\end{proof}

Let $T_0$ be a polytope resulting from the intersection of $\R^c_{>0}$
with some of the half-spaces $H_D$ defined in the above proof.  If
$T_0\supsetneq T$, some face $A=\overline{T}_0\cap \partial H_D$ of
$T_0$ will contain a point $v\in \R^c_{>0}$ not in $\overline{T}$,
and thus not in $\overline{T}\cap \partial H_D$, nor in the cone on this, 
$\overline{\P}_D$. Therefore $L_D v\neq 0$ or $F_f v < 0$ for some row $F_f$ of
$F_D$. This gives a test for when $T_0$ properly contains $T$; when
satisfied, it yields a new half-space $H_{D(v)}$ whose intersection
with $T_0$ is strictly smaller. After a finite number of intersections
we arrive at $T_0=T$. See Figure~\ref{construction}.

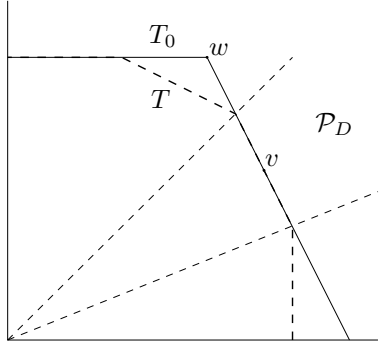
\begin{figure}[htbp]
\epsfig{file=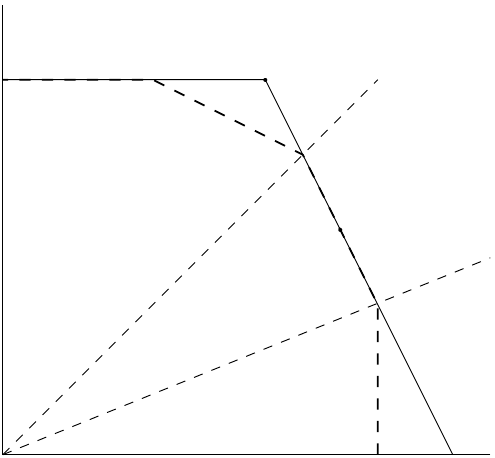, height=1.8in}
\begin{picture}(0,0)(208,0)
\put(141,95){$v$}
\put(168,114){$\P_D$}
\put(111,155){$w$}
\put(80,126){$T$}
\put(79,161){$T_0$}
\end{picture}
\caption{$T_0$ is a partially computed tilt polytope. The face
of $T_0$ containing $v\in\P_D$ has a vertex $w$ not in
$\overline{\P}_D$. Therefore $D(w)$ gives another face of $T$.} 
\label{construction}
\end{figure}

The computational geometry involved in the above is relatively
straightforward.  By using homogeneous coordinates we can treat an
unbounded region, such as $\R^c_{>0}$, as a polytope with some
vertices ``at infinity''.  The face $A$ of $T_0$, as defined above, is
the convex hull of those vertices $v$ of $T_0$ satisfying ${\Sigma}_D \cdot v = 1$.
If any of these satisfy $L_D v\neq 0$ or $F_f v < 0$ 
for some row $F_f$ of $F_D$
we conclude that $T_0\neq T$. If such $v$ lies in $\partial \R^c_{>0}$ we
perturb it a little to bring it inside $\R^c_{>0}$ before determining
a new half-space $H_{D(v)}$.

\subsection{Example: The Borromean rings complement}\label{b_rings_ex}
Let $M$ be the complement of the Borromean rings in $S^3$: \\
$$\epsfig{file=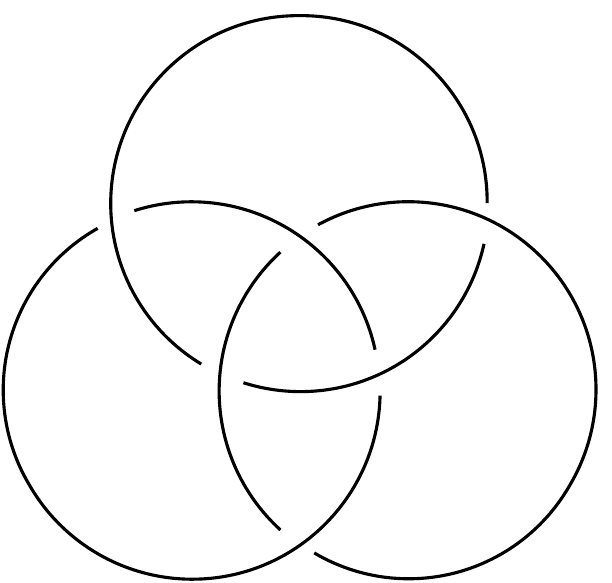, height=1in}$$
Then $M$ 
is an arithmetic hyperbolic $3$-manifold with $3$ cusps. It may be
realized by gluing two regular ideal hyperbolic octahedra in the following
pattern (see \cite{Th}).
$$
\epsfig{file=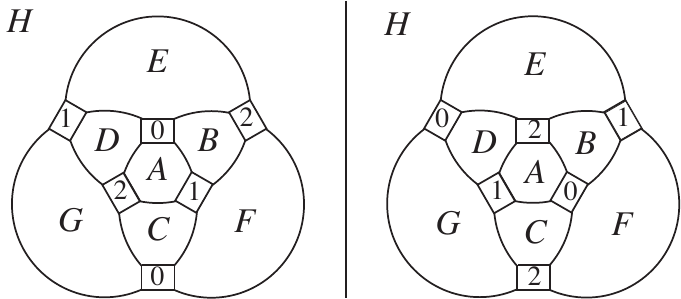, height=86.5pt}
$$
Letters indicate face identifications; cusps are numbered
$0,1,2$. Equal area cusp cross sections give rise to symmetrically
placed ideal vertex cross sections in the two octahedra. It follows
that the corresponding ideal cell decomposition consists of precisely
these two octahedra. We will first compute its parameter cell. 

Let us position the tiling of $\H^3$ by ideal octahedra such that one
of them, call it $C$, has a lift with vertices at $\{(\pm 1,0,0),
(0,\pm 1,0), (0,0,\pm 1)\}$ in the projective ball model. 
When the cusp cross sections all have equal area and contain
the centre of the octahedron, their Minkowski normals are
$\n_0\!=\!(1,0,0,1),\ \n_1\!=\!(0,1,0,1),\ \n_2\!=\!(0,0,1,1),\ 
\n_3\!=\!(-1,0,0,1),\ \n_4\!=\!(0,-1,0,1),\ \n_5\!=\!(0,0,-1,1).$
A neighbouring lift of the other octahedron,
call it $C'$, has vertex representatives $\{\n_0,\n_1,\n_2,\n_3',
\n_4',\n_5'\}$ where $\n_i'$ is the Minkowski metric reflection of
$\n_i$ in the hyperplane spanned by $\{\n_0,\n_1,\n_2\}$. 
Then  $\n_3'\!=\!(1,2,2,3),\ \n_4'\!=\!(2,1,2,3),\ \n_5'\!=\!(2,2,1,3)$. 
(See the left hand side of Figure  \ref{bor0}.)

\begin{figure}[h]
\begin{center}
\epsfig{file=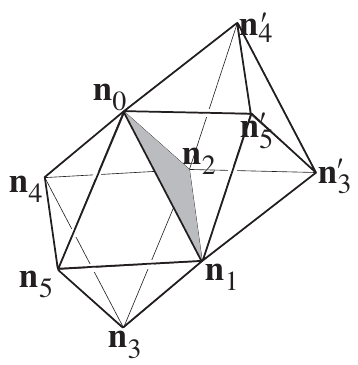} 
\hspace{.7in}
\epsfig{file=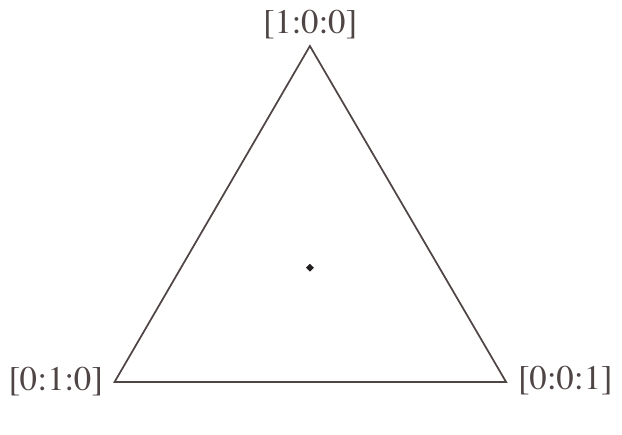}
\end{center}
\caption{Initial cell decomposition of $M$ into ideal octahedra and the corresponding 
projective parameter cell.}
\label{bor0}
\end{figure}

For size vector $v=(v_0,v_1,v_2)^t$ the resulting ideal vertex
representatives become $\n_i/v_i$, $\n_{i+3}/v_i$ and $\n_{i+3}'/v_i$ for
$i=0,1,2$. The hyperplane containing $\n_0/v_0$, $\n_1/v_1$,
$\n_2/v_2$ and $\n_3/v_0$ is $\{\x\mid \x \cdot N_v = 1\}$ where 
$N_v$ is calculated as in the proof of Proposition \ref{conditions_for_D(v)} giving
$$
N_v = \frac{1}{2}\left(
\begin{array}{rrrr}
 1 & 0 & 0 &-1 \\
-1 & 2 & 0 &-1 \\
-1 & 0 & 2 &-1 \\
 1 & 0 & 0 & 1 
\end{array}\right)
\left(\begin{array}{c} v_0 \\ v_1 \\ v_2 \\ v_0 \end{array}\right)
=
\left(
\begin{array}{rrr}
 0 & 0 & 0 \\
-1 & 1 & 0 \\
-1 & 0 & 1 \\
 1 & 0 & 0 
\end{array}\right) v.
$$
Then $\n_4/v_1$ and $\n_5/v_2$ lie in this hyperplane if and only if 
$\n_4 \cdot N_v - v_1 = 0 = (2,-2,0) v$ and $\n_5  \cdot N_v - v_2 = 0 = (2,0,-2) v$. 
Hence $L_C = 
\left(\begin{array}{rrr}
2 & -2 &  0 \\
2 &  0 & -2 
\end{array}\right)$. By symmetry, $L_{C'}$ is the same. 
Next we compute the tilt of $f_0$, the face between $C$ and
$C'$. This will be positive if $\n_3'/v_0$ lies above the hyperplane
defined by $N_v$, i.e. if $\n_3'  \cdot N_v - v_0 = (-2,2,2) v > 0$. Thus
$F_{f_0} = (-2,2,2)$, and the tilt is positive at  $v = (1,1,1)^t$. 
By symmetry, the other faces of the octahedra also have positive
tilt. The parameter cell of this decomposition $\{ v \mid v_0 = v_1 = v_2 \}$ 
is, projectively, a point (as shown in the right half of Figure \ref{bor0}).

Suppose next we increase $v_0$ slightly. The faces of $C$ still have
positive tilt. We claim that $C$ is subdivided into 4 tetrahedra containing the edge
$\n_0,\n_3$ as shown in the left of Figure \ref{bor1} below.

\begin{figure}[h]
\begin{center}
\epsfig{file=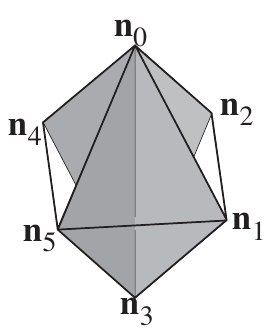} 
\hspace{.7in}
\epsfig{file=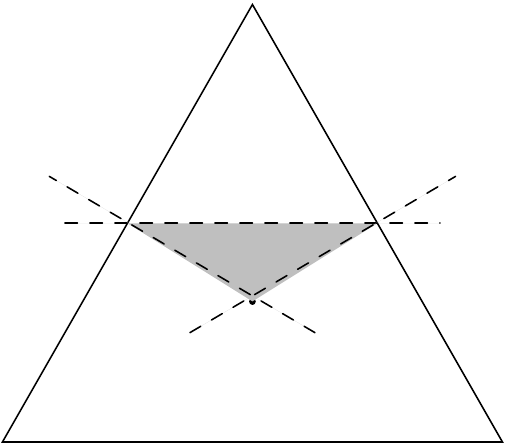, height=1.35in}
\end{center}
\caption{Decomposition into ideal tetrahedra and corresponding parameter cell.}
\label{bor1}
\end{figure}

Let $f_1$ be the triangle with vertices $\{\n_0,\n_2,\n_3\}$. Then
$f_1$ has positive tilt if and only if $\n_4 \cdot N_v - v_1 = (2,-2,0) v =
F_{f_1} v > 0$. Letting $f_2$ be the triangle with vertices
$\{\n_0,\n_1,\n_3\}$, we obtain $F_{f_2} = (2,0,-2)$. By symmetry, the
tilts of the other two faces shown above are the same. Four more
triangles, having the same tilts, subdivide $C'$. Therefore $M$ is
subdivided into $8$ simplices in the parameter cell shown in 
the right of Figure \ref{bor1}.

Increasing $v_0$, eventually $F_{f_0} v = 0$, so that the
faces of $C$ and $C'$ vanish from the decomposition. The four
simplices around the edge $\n_1,\n_2$ become an octahedron. Further
increasing $v_0$ so $F_{f_0} v < 0$, we claim that this octahedron splits at the square
dual to the edge $\n_1,\n_2$ into two square-based pyramids with vertices
$\{ \n_0, \n_0', \n_3', \n_3, \n_1\}$ and $\{ \n_0, \n_0', \n_3', \n_3, \n_2\}$
as shown in the left of Figure \ref{bor2}. 

\begin{figure}[h]
\begin{center}
\epsfig{file=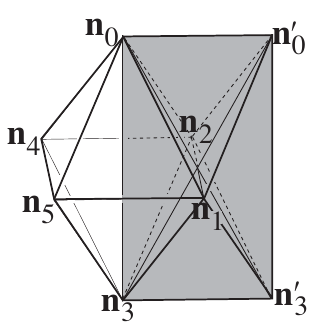} 
\hspace{.7in}
\epsfig{file=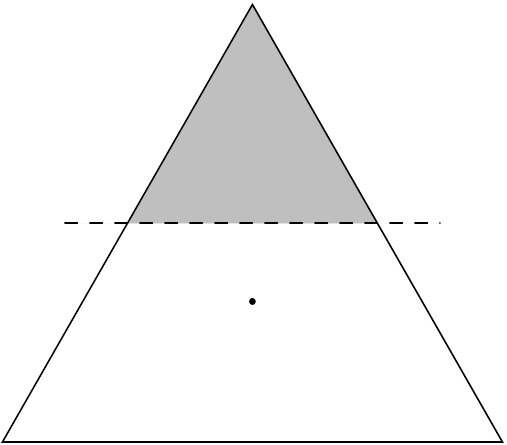, height=1.35in}
\end{center}
\caption{Decomposition into square-based pyramids and corresponding parameter cell.}
\label{bor2}
\end{figure}

Let $N_v'$ define the hyperplane containing
$\{\n_0/v_0, \n_1/v_1,\n_3/v_0, \n_3'/v_0\}$. Then 
$$
N_v' = \left(\begin{array}{rrr}
0 & 0 & 0 \\
-1& 1 & 0 \\
0 &-1 & 0 \\
1 & 0 & 0
\end{array}\right) v.
$$
Let $C_1$ be the
square-based pyramid with vertices $\{\n_0,\n_1,\n_3, \n_3',
\n_0'\}$, 
where $\n_0'\!=\!(-1,2,2,3)$ is the Minkowski reflection of $\n_3'$ in the plane
containing $\n_1,\n_2,\n_4,\n_5$. We find that $L_{C_1} =
(0,0,0)$. Let $f_3$ be the triangle with vertices $\{\n_0, \n_3,
\n_3'\}$. Then $F_{f_3} = (1,-1,-1)$ and $F_{f_3} v >0$.
Therefore $M$ is divided into $4$ square-based pyramids in
the parameter cell shown in the right of Figure \ref{bor2}, as mentioned in Remark \ref{non-generic-cell}. 

\begin{figure}[h]
\begin{center}
\epsfig{file=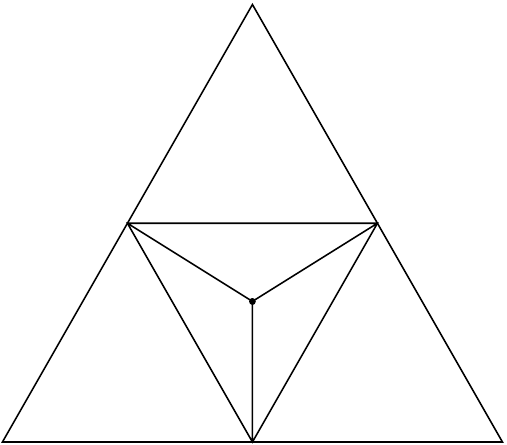, height=1.35in}
\end{center}
\caption{All parameter cells for canonical decompositions of the Borromean rings complement}
\label{bor3}
\end{figure}

By symmetry, the full set of parameter cells is projectively as shown in Figure \ref{bor3}. 
On the three line segments containing the centre point, each of the original octahedra  
is subdivided into two square based pyramids.

\section{Commensurability of cusps}\label{cuspcom}

The number of canonical cell decompositions that can arise for a
multi-cusped manifold can be quite large, placing practical limits on
the usefulness of our methods. We show next how it is often
possible to greatly reduce the complexity of computing the
commensurator of a manifold.

A horospherical cross section of a cusp in an orientable hyperbolic
$3$-manifold is a Euclidean torus, well defined up to
similarity. We can position and scale a fundamental parallelogram
in $\C$ such that one vertex lies at the origin and the two adjacent edges end
at $1$ and a point $z$ in the upper half-plane. Such a $z$ is called a
{\em cusp shape parameter.} Alternative choices of fundamental
parallelogram yield parameters differing by the action of $\mbox{\rm SL}_2\Z$
by M\"obius transformations. By choosing a suitable fundamental domain
for the action of $\mbox{\rm SL}_2\Z$ on the upper half-plane, we can
make a canonical choice of shape parameter for each cusp.

If one cusp covers another, there is an induced covering of Euclidean
tori. It follows that their cusp shape parameters are related by the action
of an element of $\mbox{\rm GL}_2\Q$. Let us call cusp shapes
{\em commensurable} if they are so related.
As we shall see shortly, it is not hard to
determine when two cusp shapes are commensurable.

Suppose we are trying to determine the (discrete) commensurator
quotient $Q$ of $M$. If $M$ has cusps of incommensurable shape, these
necessarily cover distinct cusps of $Q$. Therefore any assignment of
cusp neighbourhoods in $Q$ will yield a Ford spine and tiling whose
symmetry group is the whole commensurator (the group of covering
transformations of $Q$). It follows that we can start by choosing
arbitrary horospheres in each of a set of representatives for the
commensurability classes of cusps of $M$. Then as we vary our choices
of horosphere in the remaining cusps we will be sure to find a tiling
whose symmetry group is the commensurator of $M$. The easiest
case is when no two cusps of $M$ are commensurable. Then any choice of
horospheres at all will do.

The harder case is when $M$ has multiple commensurable cusps. If the
symmetry group of $M$ is non-trivial, it may act by non-trivially
permuting some of these cusps. Conceptually we should first divide $M$
by its symmetry group and then find the commensurator of this
quotient. For practical purposes it is hard to work with non-manifold
quotients. Instead, whenever two cusps are related by a symmetry of
$M$, we choose symmetrically equivalent cusp neighbourhoods. Let $c$ be
the number of orbits under the action of $\mbox{\rm Symm}(M)$ on the
cusps. Let $d$ be the number of distinct commensurability classes of
cusp shape. Then the parameter space of relative horosphere positions
we need to search, in order to find a tiling whose symmetry group
equals the commensurator, has dimension $c-d$.

Our algorithm is really only a slight modification of the algorithm of Weeks \cite{We1}
for finding the symmetries of a cusped hyperbolic manifold
$M=\H^3/\Gamma$. In order to find the symmetries we need only consider
the canonical cell decomposition arising from a choice of cusp neighbourhoods
such that all boundary tori have equal area. Then the symmetries of
$M$ are the symmetries of the lifted tiling that normalize $\Gamma$.
Equivalently, they are the symmetries of the tiling for which the
covering described in Remark~\ref{iscovering} has degree one.

Returning to the question of when two cusp shapes are commensurable,
we note first that cusp shapes of $M$ belong to the invariant trace
field of $M$. But if $k$ is any number field, and $\alpha, \alpha'$ are
irrational elements of $k$, they are related by an element of
$\mbox{\rm GL}_2\Q$ if and only if
\begin{equation}\label{cce}
(c\alpha + d)\alpha'=a\alpha + b
\end{equation}
is soluble for $a,b,c,d\in\Q$ such that 
\begin{equation} \label{detcond}
ad - bc\neq 0. 
\end{equation}
We can
replace (\ref{detcond}) with the condition that $a,b,c,d$ are not all zero, since
for $\alpha,\alpha'\notin\Q$, (\ref{detcond}) follows automatically
from (\ref{cce}) and the fact that $\{1,\alpha\}$ are linearly independent over $\Q$. 
Regarding $k$ as a finite dimensional vector space over $\Q$ we see that (\ref{cce})
has non trivial solutions if and only if $\{1,\alpha,
\alpha',\alpha\alpha'\}$ are linearly dependent over $\Q$.
In particular, if $[k:\Q]<4$ all irrationals are commensurable in this
sense. We thank Ian Agol for pointing out this condition.

\subsection{Example: a knot with cusp field not equal to invariant trace field}
\label{shape_not_trace}
An  interesting example uncovered during this work is the complement of the knot
$12n706$ shown below. This has one torus cusp with shape parameter
$z = 6i$ generating a cusp field $\Q(i)$ which is strictly contained in its invariant trace field $\Q(i,\sqrt{3})$. This answers a question of Neumann-Reid in \cite{NR}, who asked whether the figure eight knot and the two dodecahedral knots of Aitchison-Rubinstein \cite{AR} were the only such examples\footnote{Alan Reid informs
us that Nathan Dunfield has found another example: the 15 crossing knot
$15n{132539}$.}.
$$
\epsfig{file=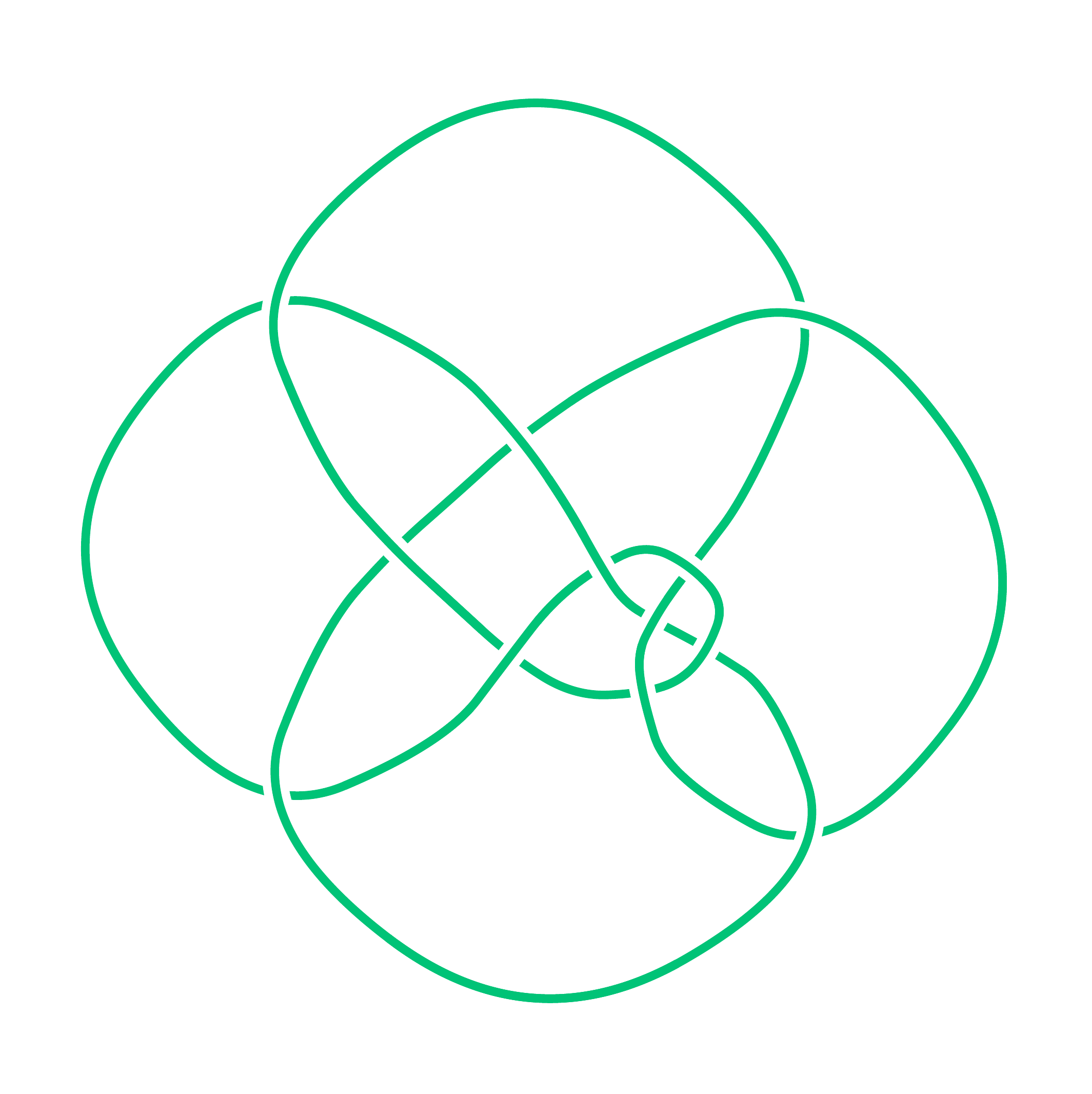, height=1.4in}
$$

\section{Experimental Results}

We have implemented the algorithms described here and used them
to compute the commensurability classes of all 4929 manifolds in
the Hildebrand-Weeks census of cusped hyperbolic $3$-manifolds
\cite{CHW} and all 7969 complements of the
8614 hyperbolic knots and links up to 12 crossings.
After replacing non-orientable manifolds (in the 5-census) by
their double covers and removing duplicates we obtained a total of
12783 orientable manifolds. Tables of all the results are available at
{\tt http://www.ms.unimelb.edu.au/\~{}snap/}.

In the process it was necessary to identify the arithmetic
manifolds and arrange them into their own separate commensurability
classes. Criteria for arithmeticity and for commensurability of arithmetic
manifolds are given in \cite{CGHN} and \cite{MacReid}. In fact, a cusped manifold is
arithmetic if and only if its invariant trace field is imaginary quadratic and
it has integer traces. Cusped arithmetic 
manifolds are commensurable if and only if they have the
same invariant trace field. They are therefore classified by the discriminant
$d$ of the invariant trace field $\Q(\sqrt{d})$. 

There were 142 arithmetic manifolds in six commensurability classes.
Table~\ref{artab} lists some of the arithmetic manifolds found with
discriminant $-3, -4$ or $-7$, and all with discriminant $-8, -11$ or
$-15$.

\begin{table}[htbp]
\begin{tabular}{ll}
$d$ & manifolds \\
\hline
-3 & \parbox{4in}{\rule{0pt}{2ex}m000,
  m002, m003, m004=4a1, m025, m203=6a5=11n318, s118, s961, 8a39,
  10a280, 10n130, 10n143, 10n155, 11n539, 12a3285, 
  \rule[-1ex]{0pt}{1ex}\ldots
} \\
\hline
-4 & \parbox{4in}{\rule{0pt}{2ex}m001,
  m124=8n10=10n139, m125, m126, m127, m128, m129, s859, v1858,
  8n8=9n34=10n112, 9n36, 10a242, 10n74,
  11n545, \rule[-1ex]{0pt}{1ex}\ldots
} \\
\hline
-7 & \parbox{4in}{\rule{0pt}{2ex}m009,
  m010, s772, s773, s774, s775, s776=6a8=8n7, 8a25, 8a37,
  10n73=12n968, 10n113, 
  11n498, 12n3068, 12n3078, 12n3093, 
  \rule[-1ex]{0pt}{1ex}\ldots
} \\
\hline
-8 & \parbox{4in}{\rule{0pt}{2ex}v2787,
  v2788,
  v2789,
  9a73,
  9a74,
 12a3292,
 12a3296,
 12n2625= 12n2630,
 12n2972,
 12n3088,
 \rule[-1ex]{0pt}{1ex}12n3098=12n3099.
} \\
\hline
-11 & \parbox{4in}{\rule[-1ex]{0pt}{3ex}12a2126,
 12a2961,
 12a3039,
 12a3230,
 12a3295.
} \\
\hline
-15 & \parbox{4in}{\rule[-1ex]{0pt}{3ex}12a3169,
 12a3273,
 12a3284,
 12a3300,
 12a3307,
 12a3308.
}
\end{tabular}
\caption{Selected arithmetic manifolds.}\label{artab}
\end{table}

The naming of manifolds in the tables is as follows. Manifolds whose names begin
with m, s or v belong to the 5, 6 or 7 tetrahedra census of cusped
manifolds respectively. The rest are knot and link complements in the
form {\it$<$number of crossings$>$ $<$alternating or non-alternating$>$ $<$index
  in table$>$.} The link tables were provided by Morwen Thistlethwaite and 
are  included with current versions of Snap \cite{Snap} and Tube \cite{Tube}.

The remaining 12641 manifolds were non-arithmetic, falling into 11709
commensurability classes. A few of them are shown in Table~\ref{nonar}.

\begin{table}[htbp]
\begin{tabular}{lrrllll}

manifold & s.g. & c.d. & c.vol,c.cl   & field  & nc,cnc & cusp density \\
\hline
10a297  &20 & 40 & 0.365076519 & 4,1025(2) & 5,1 & 0.642619992 \\
10a291  & 4 & 40 & 0.365076519 & 4,1025(2) & 4,1 & 0.642619992 \\
10a277  & 4 & 40 & 0.365076519 & 4,1025(2) & 3,1 & 0.642619992 \\
12n2492 & 4 & 40 & 0.365076519 & 4,1025(2) & 3,1 & 0.642619992 \\
12n2899 & 4 & 40 & 0.365076519 & 4,1025(2) & 4,1 & 0.642619992 \\
\hline
12n1189 & 2 & 2 & 7.175483613,0 & 7,-76154488(-2) & 2,2 & 0.589830477 \\
12n1190 & 2 & 2 & 7.175483613,1 & 7,-76154488(-2) & 2,2 & 0.589830477 \\
12n1481 & 2 & 2 & 7.175483613,2 & 7,-76154488(-2) & 2,2 & 0.589830477 \\
12n2348 & 2 & 2 & 7.175483613,3 & 7,-76154488(-2) & 3,3 & 0.644747497 \\
12n2580 & 2 & 2 & 7.175483613,4 & 7,-76154488(-2) & 3,3 & 0.631898787 \\
\hline
m045    & 4 & 4 & 0.818967911   & 3,-107(-2) & 1,1 & 0.608307263 \\
m046    & 4 & 4 & 0.818967911   & 3,-107(-2) & 1,1 & 0.608307263 \\
v3379   & 8 & 8 & 0.818967911   & 3,-107(-2) & 2,1 & 0.608307263 \\
v3383   & 4 & 8 & 0.818967911   & 3,-107(-2) & 3,1 & 0.608307263 \\
v3384   & 8 & 8 & 0.818967911   & 3,-107(-2) & 2,1 & 0.608307263 \\
12a1743 & 8 &16 & 0.818967911   & 3,-107(-2) & 2,1 & 0.608307263 \\
v3376   & 4 & 4 & 1.637935822,0 & 3,-107(-2) & 2,1 & 0.690189995 \\
v3377   & 4 & 4 & 1.637935822,0 & 3,-107(2)  & 1,1 & 0.690189995 \\
v3378   & 4 & 4 & 1.637935822,0 & 3,-107(2)  & 1,1 & 0.690189995 \\
12a2937 & 8 & 8 & 1.637935822,1 & 6,-1225043(1) & 3,2 & 0.608307263 \\
\hline
9a94 & 4 & 24 & 0.575553268 & 4,144(1) & 3,2 & 0.844133714
\end{tabular}
\caption{Selected non-arithmetic manifolds.} \label{nonar}
\end{table}

Column headings are as follows: s.g. is the order of the symmetry
group; c.d. is the degree of the manifold over its commensurator
quotient; c.vol is the volume of the commensurator quotient, optionally
followed by c.cl,
commensurability class numbered from 0, when incommensurable
manifolds are listed with the same commensurator volume. Thus manifolds are
commensurable if and only if they have the same entry in this column. 
The invariant trace field is described by its degree, discriminant and a number
specifying which root of the minimum polynomial generates it (with
sign corresponding to choice of complex conjugate); nc,cnc gives the
number of cusps in the manifold and the number of cusps in the
commensurator quotient; cusp density is computed using equal area cusps in the
commensurator quotient. 

The first group of manifolds have the smallest commensurator volume
found (among non-arithmetic manifolds) and are the link complements
which appeared in Section~\ref{5_links}. They all have `hidden
symmetries,' i.e. commensurabilities not arising from the symmetry group of
the manifold. The total number of non-arithmetic manifolds having
hidden symmetries was 148. The next group of manifolds shows that 
incommensurable
manifolds are not always distinguished by cusp density. The third
group of manifolds includes manifolds whose classes are distinguished
by invariant trace field but not by cusp density, and manifolds with the same
commensurator quotient volume but different invariant trace fields. 
One should not get the impression that cusp density is a poor invariant: 
in fact, among the 11278 cusp densities found, only 417 grouped together incommensurable manifolds. 
The final line gives data for the non-arithmetic manifold with highest cusp density found.
(The maximum possible cusp density is $0.853276\ldots$, which occurs for the figure eight knot complement.)

More details of the fields occurring above are listed in the Table \ref{fields}.

\begin{table}[htbp]
\begin{tabular}{crcl}

degree & discriminant & signature & minimum polynomial \\
\hline
4 & 1025 & 0, 2 & $x^4 - x^3 + 3x^2 - 2x + 4$\\
7 &  -76154488 & 1, 3 & $x^7 - 2x^6 + 3x^5 - 5x^4 + x^3 - 8x^2 - 2x - 4$\\
3 & -107 & 1, 1 & $x^3 - x^2 + 3x - 2$\\
6 & -1225043 & 0, 3 & $x^6 - 2x^5 - 2x^3 + 30x^2 - 52x + 29$\\
4 & 144 & 0, 2  & $x^4 - x^2 + 1$\\
\end{tabular}
\caption{Fields in previous table.} \label{fields}
\end{table}

\section{Appendix}

The indexing system used for knots and links, here and in {\tt snap},
may still be subject to change. This is due to the difficulty of
determining whether two non-hyperbolic links are equivalent and the
consequent possibility that duplicates will later be discovered and 
removed from the tables.

For this reason we provide Table~\ref{dowker_codes}, giving the
Dowker-Thistlethwaite codes of all the knots and links that we refer
to. Since, for links, this code is not well known, we
describe here how this works. 

Firstly there is the trivial matter of passing between the alphabetic
codes used by {\tt snap}, and their numerical forms. The Dowker
code for a link with $n$ crossings and $k$ components is a
permutation of the even integers $2,\ldots, 2n$, with possible sign
changes, bracketed into $k$ subsequences: e.g. $(6,-8)\
(-10,14,-12,-16,-2,4)$. To express this alphabetically we encode $n$
and $k$ as
the first two letters
using $a=1,b=2$, etc. Then follow $k$ letters
giving the lengths of the bracketed subsequences. Finally there are
$n$ letters giving the sequence of even integers using
$a=2,b=4,$ etc.\ and $A=-2,B=-4,$ etc. The alphabetic code for the above
example is thus {\tt hbbfcDEgFHAb}. 

To go from a link diagram to its Dowker code, proceed as follows.
Traverse each component, numbering the crossings, starting with $1$ on an
overcrossing. When the first component is done, continue with
consecutive numbers on the next component. Each crossing will receive
two numbers. We can number the crossings in such a way that every crossing gets one
even and one odd number. (This follows easily from the fact that we
can two-colour the plane containing a link diagram.) Negate any even
number which labels an over-crossing (for an alternating link there
will not be any).  For each odd number $1,3, \ldots, ,2n-1$ write down the
corresponding even number: this gives a sequence of $n$ even numbers.  A
component with $2j$ crossings will have $j$ odd numbers on it,
so there will be $j$ corresponding numbers for it in the code; bracket
together the numbers for each component. See for example
Figure~\ref{dcexample}. 

\begin{figure}[htbp]
$$
\epsfig{file=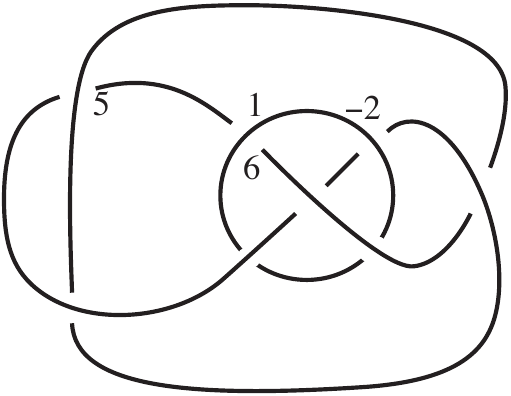,height=3cm} 
$$
\caption{Complete the numbering indicated on the above two-component
link to obtain the Dowker-Thistlethwaite code $(6,-8)\ (-10,14,-12,-16,-2,4)$.}
\label{dcexample}
\end{figure}

The reverse procedure, going from a code to a link diagram is a little
bit more tricky, but essentially the same as the procedure
for knots, described in \cite{Ad}. We draw the first component as a
knot, with extra as-yet-unconnected crossings on it. When adding
further components we will encounter crossings with other link
components. 

 \begin{table}[htbp]
\begin{tabular}{llll}
name & Dowker code & name & Dowker code \\
4a1     & dadbcda &
12a2937 & lcbeechjklaiefgbd \\
6a5     & fbccdefacb &
12a2961 & lcbfdcfahibekdlgj \\
6a8     & fcbbbceafbd &
12a3039 & lcbhbcfaikbjelgdh \\
8a25    & hbbfcfegahdb &
12a3169 & lccfcdfhcjakblieg \\
8a37    & hcbcccfghadeb &
12a3230 & ldbccdcfagiebkdlhj \\
8a39    & hdbbbbceagbhdf &
12a3273 & ldbddbceagbidkflhj \\
8n7     & hcbcccDFgHEaB &
12a3284 & ldbfbbceaibkjldgfh \\
8n8     & hcbcccdFgHEab &
12a3285 & ldbfbbcfaikbjlehdg \\
8n10    & hdbbbbcEaGBHDF &
12a3292 & ldccccdegjhkflacib \\
9a73    & ibcfdfhagbice &
12a3295 & ldccccdgjakhbflice \\
9a74    & ibcfdfhaibecg &
12a3296 & ldccccdgjhkbelcfia \\
9a94    & icbebcdhfiabeg &
12a3300 & lebbbdbceagbidkflhj \\
9n34    & icbdcceaGbHDIF &
12a3307 & lebccbbcfaikbljdheg \\
9n36    & icbebcdHfiabeG &
12a3308 & lfbbbbbbceagbidkflhj \\
10a242  & jbeefghjiaecbd &
12n706  & lalceFhGjIKlBaD \\  
10a277  & jcbfbceagbidjfh &
12n968  & lbbjcDFhIJLAbKEG \\
10a280  & jcccddeghjibcfa &
12n1189 & lbbjceaHbKiDjgLF \\
10a291  & jdbbdbceagbidjfh &
12n1190 & lbbjceaHbkIDJGlf \\
10a297  & jebbbbbceagbidjfh &
12n1481 & lbbjchaEGJDbkFli \\
10n73   & jbbhcEFihGJAdb &
12n1848 & lbcidfIaglceKBHJ \\
10n74   & jbbhcfaHGIbDJE &
12n2348 & lcbbhcEaHBKiDjgLF \\
10n112  & jcbcecdFgIHabJE &
12n2492 & lcbdfceaGbiDJKfLH \\
10n113  & jcbcecFaHJIBDGE &
12n2580 & lcbeecHaEGJDBkFli \\
10n130  & jcbfbceaGbIDJFH &
12n2625 & lcbfdcfaIJbKLEDHG \\
10n139  & jdbbcccEaHBijDgf &
12n2630 & lcbfdcfaIJbKLHDEG \\
10n143  & jdbbdbcEaGBiDjfh &
12n2899 & ldbbbfcEaGBIDjkFlh \\
10n155  & jebbbbbcEaGBIDjFh &
12n2972 & ldbccdcfaIJbKLEDHG \\
11n318  & kbchdEfcHiJKaBG &
12n3068 & ldbddbceaGbIDKFLHJ \\
11n498  & kcbcfcfaHJbIKEGD &
12n3078 & ldbfbbceaIbKJLDGFH \\
11n539  & kcbedceaHbIJDGKF &
12n3088 & ldccccdeGacJBklFih \\
11n545  & kcbfccdIfJabKGEH &
12n3093 & ldccccdEGHiJKlAFBc \\
12a1743 & lbbjchfjialkedbg &
12n3098 & ldccccdEGhIJKLaFBC \\
12a2126 & lbcidgjclhafkbie & 
12n3099 & ldccccdEGhIJKlAFBc 
\end{tabular}
\caption{Dowker-Thistlethwaite codes of all knots and links mentioned
in this paper.}\label{dowker_codes}
\end{table}


\newpage

\begin{thebibliography}{[12]}

\bibitem{Ad1} C. Adams, 
Noncompact hyperbolic $3$-orbifolds of small volume, 
Topology '90 (Columbus, OH, 1990), 1--15,
Ohio State Univ. Math. Res. Inst. Publ., 1, de Gruyter, Berlin, 1992.

\bibitem{Ad} C. Adams,
The Knot Book, W. H. Freeman and Company, New York, 1994.

\bibitem{AR} I. R. Aitchison and J. H. Rubinstein, Combinatorial
Cubings, Cusps and the Dodecahedral Knots, 
Topology '90 (Columbus, OH, 1990), 273--310,
Ohio State Univ. Math. Res. Inst. Publ., 1, de Gruyter, Berlin, 1992.

\bibitem{Ak} H. Akiyoshi,
Finiteness of polyhedral decompositions of cusped hyperbolic manifolds
obtained by the Epstein-Penner's method.
Proc. Amer. Math. Soc. 129 (2001), no. 8, 2431--2439.

\bibitem{ASWY} H. Akiyoshi, M. Sakuma, M. Wada and Y. Yamashita,
Punctured Torus Groups and 2-Bridge Knot Groups (I),
Lecture Notes in Mathematics, Vol. 1909, Springer, 2007.

\bibitem{Bor} A. Borel, Commensurability classes and volumes of hyperbolic
$3$-manifolds. Ann. Scuola Norm. Sup. Pisa Cl. Sci. (4) 8 (1981), no. 1,
1--33.

\bibitem{BMR} B. H. Bowditch, C. Maclachlan and A. W. Reid,
Arithmetic hyperbolic surface bundles.
Math. Ann. 302 (1995), no. 1, 31--60.

\bibitem{Bu} J. O. Button, 
Fibred and virtually fibred hyperbolic 3-manifolds in the censuses. 
Experiment. Math. 14 (2005), no. 2, 231--255.

\bibitem{CHW}
P. J. Callahan, M. V. Hildebrand and J. R. Weeks,
A census of cusped hyperbolic $3$-manifolds.
(With microfiche supplement.)
Math. Comp. 68 (1999), no. 225, 321--332.

\bibitem{CGHN}  D. Coulson, O. Goodman,  C. Hodgson
and W. Neumann, Computing arithmetic invariants of
3-manifolds, Experimental Mathematics {\bf 9} (2000),
127--152.

\bibitem{EP} D. B. A. Epstein and  R. C. Penner,
Euclidean decompositions of noncompact hyperbolic manifolds.
J. Differential Geom. 27 (1988), no. 1, 67--80.

\bibitem{FH} W.~Floyd and A.~Hatcher,
Incompressible surfaces in punctured-torus bundles.
Topology and its Applications 13 (1982), 263--282.

\bibitem{GW} F.~Gonz\'{a}lez-Acu\~{n}a and W. C. Whitten,
Imbeddings of three-manifold groups. 
Mem. Amer. Math. Soc. 99 (1992), no. 474, viii+55 pp.

\bibitem{GF} F. Gu\'eritaud (with an appendix by D. Futer), 
On canonical triangulations of once-punctured torus bundles and two-bridge link complements, 
Geometry and Topology 10 (2006), pp. 1239-1284.

\bibitem{Guer} F. Gu\'eritaud, G\'eom\'etrie hyperbolique effective et triangulations id\'eals canoniques en dimension 3, Th\`ese de doctorat,  Univ. Paris-sud, Orsay, 2006.

\bibitem{Snap}
O. Goodman, {\em Snap, the computer program},
{http://www.ms.unimelb.edu.au/\~{}snap/}
and {http://sourceforge.net/projects/snap-pari}.

\bibitem{Tube} O. Goodman, {\em Tube: a computer program for studying
geodesics in hyperbolic $3$-manifolds}, available from
{http://www.ms.unimelb.edu.au/\~{}snap/}.

\bibitem{Orb} D. Heard, {\em Orb: a computer program for studying hyperbolic
3-orbifolds}, available from
{http://www.ms.unimelb.edu.au/\~{}snap/orb.html}.

\bibitem{HiW}
M. Hildebrand and J. Weeks,
A computer generated census of cusped hyperbolic $3$-manifolds.
Computers and mathematics (Cambridge, MA, 1989), 53--59,
Springer, New York, 1989.

\bibitem{knotscape} J.~Hoste and M.~Thistlethwaite, {\em Knotscape},
{http://www.math.utk.edu/\~{}morwen/knotscape.html}.

\bibitem{La} M. Lackenby,
The canonical decomposition of once-punctured torus bundles, 
Comment. Math. Helv. 78 (2003) 363--384.

\bibitem{Ma}  A. M. Macbeath, Commensurability of co-compact
three-dimensional hyperbolic groups. Duke Math. J. 50 (1983), no. 4,
1245--1253.  Erratum:  Duke Math. J. 56 (1988), no. 1, 219.

\bibitem{MacReid} C. Maclachlan and A. Reid, 
The arithmetic of hyperbolic 3-manifolds, 
Springer-Verlag, New York, 2003.

\bibitem{Marg} G. A. Margulis, Discrete subgroups of semisimple Lie groups,
{\em Ergebnisse der Mathematik und ihrer Grenzgebiete 17}, Springer-Verlag, 
Berlin, 1991. 

\bibitem{MM} T. H. Marshall and G. J. Martin, 
Minimal Co-Volume Hyperbolic Lattices II,  in preparation.

\bibitem{Me} R. Meyerhoff, The cusped hyperbolic 3-orbifold of minimal volume, Bull. Amer. Math. Soc. 13 (1985), 154--156.

\bibitem{NR}
W.~D.~Neumann and A.~W.~Reid, Arithmetic of Hyperbolic Manifolds, 
Topology '90 (Columbus, OH, 1990), 273--310,
Ohio State Univ. Math. Res. Inst. Publ., 1, de Gruyter, Berlin, 1992.

\bibitem{NR1}
W.~D.~Neumann and A.~W.~Reid, Notes on Adams' small volume orbifolds, 
Topology '90 (Columbus, OH, 1990), 311--314,
Ohio State Univ. Math. Res. Inst. Publ., 1, de Gruyter, Berlin, 1992.

\bibitem{R} A. W. Reid,
A note on trace-fields of Kleinian groups.
Bull. London Math. Soc. 22 (1990), no. 4, 349--352.

\bibitem{SW1} M. Sakuma and J. Weeks,
Examples of canonical decompositions of hyperbolic link complements.
Japan. J. Math. (N.S.) 21 (1995), no. 2, 393--439.

\bibitem{SW2} M. Sakuma and J. Weeks,
The generalized tilt formula.
Geom. Dedicata 55 (1995), no. 2, 115--123.

\bibitem{Th} W.~P.~Thurston,
The geometry and topology of three-manifolds, Princeton University
Math. Dept. (1978). Also available at
{http://msri.org/publications/books/gt3m}.

\bibitem{We1} J. R. Weeks,
Convex hulls and isometries of cusped hyperbolic $3$-manifolds.
Topology Appl. 52 (1993), no. 2, 127--149.

\bibitem{We}
J. Weeks, {\em SnapPea, the computer program}, available
from\hfill\break
{http://geometrygames.org/SnapPea/index.html}.

\bibitem{Zi} R. Zimmer, Ergodic theory and semi-simple Lie groups,
Birkhauser, Boston, 1984.

\end{thebibliography}
\end{document}